\documentclass[11pt]{amsart}
\usepackage{amsmath,amsfonts,mathrsfs}
\usepackage{amssymb}

\usepackage[margin=1in]{geometry}
\usepackage{url}
\usepackage[shortlabels]{enumitem}

\usepackage[unicode,bookmarks]{hyperref}
\usepackage[usenames,dvipsnames]{xcolor}
\hypersetup{colorlinks=true,citecolor=NavyBlue,linkcolor=Brown,urlcolor=Orange}

\usepackage{tikz-cd,adjustbox}
\usepackage[arrow,curve,matrix]{xy}

\newif\ifdraft
\draftfalse

\newcommand{\tensor}{\otimes}
\newcommand{\isom}{\simeq}

\newcommand{\C}{\mathbb{C}}

\newcommand{\Q}{\mathbb{Q}}
\newcommand{\Z}{\mathbb{Z}}

\newcommand{\DD}{\mathbb{D}}

\newcommand{\M}{\mathcal{M}}
\newcommand{\N}{\mathcal{N}}
\newcommand{\D}{\mathcal{D}}

\renewcommand{\H}{\mathcal{H}}
\newcommand{\HH}{\mathbb{H}}

\renewcommand{\O}{\mathcal{O}}
\newcommand{\DR}{\mathrm{DR}}

\newcommand{\codim}{\mathrm{codim}}

\newcommand{\K}{\mathcal{K}}
\newcommand{\gr}{\mathrm{gr}}

\newcommand{\h}{\underline h}

\newcommand{\DB}{\underline{\Omega}} %for du Bois complex
\newcommand{\Sing}{\mathrm{Sing}}
\newcommand{\dual}{\mathbf{D}}
\newcommand{\Supp}{\mathrm{Supp}}

\newcommand{\lcdef}{\mathrm{lcdef}}
\newcommand{\depth}{\mathrm{depth}}
\newcommand{\sm}{\smallsetminus}

\newcommand{\IC}{\mathrm{IC}}

\counterwithin{equation}{subsection}
\counterwithout{subsection}{section}
\counterwithin{figure}{subsection}

\newtheorem{thm}[equation]{Theorem}
\newtheorem{cor}[equation]{Corollary}
\newtheorem{lem}[equation]{Lemma}
\newtheorem{prop}[equation]{Proposition}

\theoremstyle{definition}
\newtheorem{defn}[equation]{Definition}

\theoremstyle{remark}
\newtheorem{rmk}[equation]{Remark}
\newtheorem{ex}[equation]{Example}

\theoremstyle{plain}

\newcommand{\theoremref}[1]{\hyperref[#1]{Theorem~\ref*{#1}}}
\newcommand{\lemmaref}[1]{\hyperref[#1]{Lemma~\ref*{#1}}}
\newcommand{\definitionref}[1]{\hyperref[#1]{Definition~\ref*{#1}}}
\newcommand{\propositionref}[1]{\hyperref[#1]{Proposition~\ref*{#1}}}
\newcommand{\conjectureref}[1]{\hyperref[#1]{Conjecture~\ref*{#1}}}
\newcommand{\corollaryref}[1]{\hyperref[#1]{Corollary~\ref*{#1}}}
\newcommand{\exampleref}[1]{\hyperref[#1]{Example~\ref*{#1}}}

\makeatletter
\let\old@caption\caption
\renewcommand*{\caption}[1]{%
	\setcounter{figure}{\value{equation}}%
	\stepcounter{equation}%
	\old@caption{#1}\relax%
}
\makeatother

\newcounter{intro}

\newtheorem{intro-conjecture}[intro]{Conjecture}
\newtheorem{intro-corollary}[intro]{Corollary}
\newtheorem{intro-theorem}[intro]{Theorem}

\makeatletter
\def\Ddots{\mathinner{\mkern1mu\raise\p@
\vbox{\kern7\p@\hbox{.}}\mkern2mu
\raise4\p@\hbox{.}\mkern2mu\raise7\p@\hbox{.}\mkern1mu}}
\makeatother

\begin{document}

\title[Hodge symmetry and Lefschetz theorems for singular varieties]{Hodge symmetry and Lefschetz theorems for singular varieties}

\author{Sung Gi Park}
\address{Department of Mathematics, Princeton University, Fine Hall, Washington Road, Princeton, NJ 08544, USA}
\address{Institute for Advanced Study, 1 Einstein Drive, Princeton, NJ 08540, USA}
\email{sp6631@princeton.edu \,\,\,\,\,\,sgpark@ias.edu}

\author{Mihnea Popa}
\address{Department of Mathematics, Harvard University, 1 Oxford Street, Cambridge, MA 02138, USA}
\email{mpopa@math.harvard.edu}

\thanks{MP was partially supported by NSF grants DMS-2040378 and DMS-2401498.}


\date{\today}

\begin{abstract}
We prove new results concerning the topology and Hodge theory of singular varieties. A common theme is that concrete conditions on the complexity of the singularities, from a number of different perspectives, are closely related to the symmetries of the Hodge-Du Bois diamond. We relate this to the theory of rational homology manifolds, and characterize these among low-dimensional varieties with rational singularities.
\end{abstract}

\maketitle

\makeatletter
\newcommand\@dotsep{4.5}
\def\@tocline#1#2#3#4#5#6#7{\relax
  \ifnum #1>\c@tocdepth % then omit
  \else
    \par \addpenalty\@secpenalty\addvspace{#2}%
    \begingroup \hyphenpenalty\@M
    \@ifempty{#4}{%
      \@tempdima\csname r@tocindent\number#1\endcsname\relax
    }{%
      \@tempdima#4\relax
    }%
    \parindent\z@ \leftskip#3\relax
    \advance\leftskip\@tempdima\relax
    \rightskip\@pnumwidth plus1em \parfillskip-\@pnumwidth
    #5\leavevmode\hskip-\@tempdima #6\relax
    \leaders\hbox{$\m@th
      \mkern \@dotsep mu\hbox{.}\mkern \@dotsep mu$}\hfill
    \hbox to\@pnumwidth{\@tocpagenum{#7}}\par
    \nobreak
    \endgroup
  \fi}
\def\l@section{\@tocline{1}{0pt}{1pc}{}{\bfseries}}
\def\l@subsection{\@tocline{2}{0pt}{25pt}{5pc}{}}
\makeatother

%========================================================

\tableofcontents

\section{Introduction}\label{scn:intro}
This paper is concerned with new results regarding the topology and Hodge theory of complex varieties, contributing towards a deeper understanding of the behavior of the (appropriate) Hodge diamond of a singular projective variety.
We are especially interested in its symmetries, as a concrete reflection of the complexity of the singularities. In particular, we aim for criteria for singular cohomology groups to carry pure Hodge structure, satisfy Poincar\'e duality, or better, coincide with intersection cohomology.

All the varieties we consider in this paper are over $\C$. Given a singular projective variety $X$ of dimension $n$, the ordinary cohomology $H^{p+q}(X,\Q)$ has a mixed Hodge structure with Hodge filtration $F_\bullet$; here we use the convention where it is increasing. As in the case of smooth projective varieties, the Hodge numbers 
$\underline h^{p,q}$ are defined by
$$
\underline h^{p,q}(X):=\dim_\C \gr^F_{-p}H^{p+q}(X,\C).
$$
Due to Deligne-Du Bois theory, these are also known to be computed by the hypercohomologies of the Du Bois complexes $\DB_X^p$, in the sense that 
$$
H^{p,q} (X) := \gr^F_{-p}H^{p+q}(X,\C) = \HH^q(X,\DB_X^p).
$$
Therefore, to emphasize the singular setting, we sometimes refer to them as \emph{Hodge-Du Bois numbers}. As in the smooth setting, they organize themselves in a \emph{Hodge-Du Bois diamond}; see the figure in Section \ref{scn:HDB} for the convention.
It is natural to ask under what condition on the singularities we have symmetries that are familiar in the smooth setting; for instance when, in (very roughly) increasing order of strength, for some given $i$, or $p$ and $q$, we have:

\smallskip
\noindent 
(i)  purity of the Hodge structure or Poincar\'e duality for the singular cohomology $H^i( X, \Q)$.

\smallskip
\noindent
(ii) the Hodge symmetry
$$
\underline h^{p,q}(X)=\underline h^{q,p}(X)=\underline h^{n-p,n-q}(X)=\underline h^{n-q,n-p}(X).
$$

\smallskip
\noindent
(iii) $H^{p,q} (X) \simeq \mathit{IH}^{p,q} (X)$, the corresponding Hodge spaces for intersection cohomology.

\smallskip

We focus especially on (ii) in the statements, most often having to deal technically with (i) or (iii) in the process. As we will see later, the paper \cite{FL22} provides a very nice source of inspiration for thinking about the problem via (ii). 

Note that some symmetries of this type hold for relatively simple reasons: a standard fact is that if $\dim \Sing(X) = s$, then the Hodge structure on  $H^k (X, \Q)$ is pure for $k > n + s$; see \cite[Theorem 6.33]{PS08}. We thus have the symmetry 
$$\h^{p, q} (X) = \h^{q, p} (X) \,\,\,\,\,\,{\rm for ~all}\,\,\,\, p + q > n +s.$$
Beyond this, things become quite subtle. We will see that many different ingredients go into answering the question, according to the location under consideration on the diamond. As a toy example which is already non-trivial, we consider the boundary of the diamond: it can be shown that it has full symmetry under the assumption of rational singularities, i.e.
$$
X ~ has~rational~singularities \implies  \underline h^{0,q}(X)=\underline h^{q,0}(X)=\underline h^{n,n-q}(X)=\underline h^{n-q,n}(X) \,\,\,\,{\rm for} \,\,0\le q\le n. 
$$
The first two equalities follow from basic properties of the Du Bois complex of a rational singularity, but the last is new and part of a general result, Theorem \ref{thm:Hodge-symmetry-krat} below.

There are various ways of taking this analysis deeper into the interior of the Hodge-Du Bois diamond, based on the following invariants that turn out to be fundamental for 
the Hodge theory of singular varieties:

%\smallskip
%\noindent
%(i) the defect of $\Q$-factoriality of $X$.

\smallskip
\noindent
(i) the higher rational singularities level of $X$.

\smallskip
\noindent 
(ii) the local cohomological defect of $X$.

\smallskip
\noindent
(iii) the defect of (local, analytic) $\Q$-factoriality of $X$.

\smallskip
These invariants measure different things, but there are also subtle relationships between them that will emerge throughout. In this paper we focus especially on (i) and (ii), while 
in the companion paper \cite{PP25} we address the connection with (iii) and its consequences.\footnote{Note that this paper and \cite{PP25} appeared originally as a single paper 
\cite{PP}. Both contain new material, compared to the original preprint.}

\smallskip
\noindent
{\bf Symmetries via higher rationality.}
Friedman and Laza introduced the notion of \emph{higher rational singularities}  in \cite{FL24, FL22} and noted that it is well suited for Hodge symmetries.
In \cite[Theorem 3.24]{FL22}, they proved that when $X$ is a local complete intersection (LCI) with $m$-rational singularities, one has $\underline h^{p,q} (X) =\underline h^{q,p} (X) =\underline h^{n-p,n-q} (X)$, for all $0\le p\le m$ and $0\le q\le n$. Employing the same techniques, this was later generalized to arbitrary $X$ with only normal pre-$m$-rational singularities (see Definition \ref{definition:pre-k-DB-rational}) in \cite[Corollary 4.1]{SVV23}.  However, the number $\underline h^{n-q,n-p} (X)$ remained mysterious, and is indeed much more delicate to analyze.  We complete the picture in order to obtain full symmetry, answering a question of Laza:

\begin{intro-theorem}\label{thm:Hodge-symmetry-krat}
Let $X$ be a normal projective variety of dimension $n$, with pre-$m$-rational singularities. Then
$$
\underline h^{p,q} (X) =\underline h^{q,p} (X) =\underline h^{n-p,n-q} (X) =\underline h^{n-q,n-p} (X)
$$
for all $0\le p\le m$ and $0\le q\le n$. 

In particular, if $m \ge \frac{n-2}{2}$, then we have the full symmetry of the Hodge-Du Bois diamond; in this case, $X$ is in fact a rational homology manifold. 
\end{intro-theorem}

While we state it in this form for simplicity, the result -- just like any result in this paper involving (pre-)$m$-rationality -- holds under the significantly weaker technical hypothesis that the canonical map $\DB_X^p \to I \DB_X^p$ between the respective Du Bois and intersection Du Bois complexes of $X$ is an isomorphism for $p \le m$. This is condition 
$D_m$, discussed in Section \ref{scn:IC}, known to be implied by pre-$m$-rationality by a result from \cite{PSV24}. The conclusion is also stronger; see  Theorem \ref{thm: Hodge symmetry for k-rational}.  An equivalent condition is also considered in detail, and analyzed from a different perspective, in the more recent \cite{DOR}.

The last statement of Theorem \ref{thm:Hodge-symmetry-krat} is part of a more general result, Corollary \ref{cor: lcdef of k-rational}, which is a useful byproduct of the techniques of this paper:

\noindent
\emph{A variety satisfying condition $D_m$ is a rational homology manifold in codimension $2m+2$.}

An important point regarding these results can be roughly summarized as follows:  condition $D_m$ implies the coincidence between the cohomology groups $H^\bullet (X, \Q)$, and the respective intersection cohomology groups $\mathit{IH}^\bullet (X, \Q)$, in a certain range. Since intersection cohomology is well behaved, in this range our problem is trivial. The new and more subtle part is that, thanks to deep duality features in the theory of mixed Hodge modules, the actual range of good behavior is about twice as large as the naive one. The same applies in the case of Theorem \ref{thm:Lefschetz-hyp-krat} below.

\smallskip
\noindent
{\bf Symmetries via weak Lefschetz theorems.}
The objects and methods we appeal to in the course of proving Theorem \ref{thm:Hodge-symmetry-krat} also lead to new Lefschetz hyperplane theorems for singular (ambient) varieties. This subject has seen a number of developments, the best known appearing in the book \cite{GM88} by Goresky and MacPherson; see \cite[3.1.B]{Lazarsfeld} for a review of the main results.  We propose some stronger results for cohomology groups with rational coefficients.

The first is a simple, yet powerful, application of generalized Artin vanishing for constructible complexes, as in \cite{BBD}, and the Riemann-Hilbert correspondence for local cohomology as in \cite{BBLSZ} and \cite{RSW23}, recasting \cite{Ogus73}. (For a variant for general hyperplane sections, using a vanishing theorem in the coherent category from \cite{PS24}, see Remark \ref{rmk:DB-vanishing}.) Thus, this result should be seen as the outcome of the work of a number of authors, starting with Ogus. 

\begin{intro-theorem}\label{thm:Lefschetz-lcdef}
Let $X$ be an equidimensional projective variety of dimension $n$,  $D$ be an ample effective Cartier divisor on $X$, and $U = X \smallsetminus D$. Then the restriction map
$$
H^i(X,\Q)\to H^i(D,\Q)
$$
is an isomorphism for $i\le n-2-\lcdef(U)$ and injective for $i=n-1-\lcdef(U)$.
\end{intro-theorem}

Here $\lcdef(U)$ is the \emph{local cohomological defect} of $U$; see Section \ref{scn:lcdef}. It is equal to $0$ in the LCI case, and more generally we have the inequality $\lcdef (U) \le s (U)$, where $s(U)$ is the difference between the minimal number of (local) defining equations for $U$ and its codimension, for an embedding into a smooth variety. There are however plenty of classes of non-LCI varieties with defect equal to $0$: some examples are quotient singularities, or Cohen-Macaulay varieties of dimension up to three. A more complete list can be found in Example \ref{ex:def=0}, and its extensions in the Appendix. Thus Theorem \ref{thm:Lefschetz-lcdef} is a strict improvement of the well-known result of Goresky-MacPherson \cite[Section 2.2.2]{GM88} (which is in terms of $s (U)$) when cohomology with rational coefficients is considered.\footnote{Note that the results of \cite{GM88} hold with integral coefficients, and also for higher homotopy groups.} Its consequences to the purity of Hodge structures lead to a range of  ``horizontal" symmetry in the Hodge-Du Bois diamond based on the commutative algebra properties of $X$. Concretely, 
Corollary \ref{cor:purity} states that when $X$ is regular in codimension $k$,  the Hodge structure on $H^i (X, \Q)$ is pure for $i < k - \lcdef (X)$.

It turns out that Theorem \ref{thm:Lefschetz-lcdef} does not capture the full picture. Further information is obtained by considering higher rational singularities as well. While having (pre-)$m$-rational singularities imposes subtle restrictions on $\lcdef (X)$ (see Corollary \ref{cor: lcdef of k-rational}), which brings Theorem  \ref{thm:Lefschetz-lcdef}  into play, the next result says more starting in dimension $4$, as explained in Section \ref{scn:WL2}.

\begin{intro-theorem}\label{thm:Lefschetz-hyp-krat}
Let $X$ be an equidimensional projective variety of dimension $n$, and let $D$ be an ample effective Cartier divisor on $X$. If $U = X \sm D$ has normal pre-$m$-rational singularities, then the restriction map
$$
H^{p,q}(X)\to H^{p,q}(D)
$$
is an isomorphism when $p+q\le n-2$ and  $\min\left\{p,q\right\}\le m$. Furthermore, it is injective when $p+q\le n-1$ and $\min\left\{p,q-1\right\}\le m$.
\end{intro-theorem}

We emphasize that this result already contains new information for arbitrary varieties (take $m = -1$), or for varieties with rational singularities (take $m = 0$), for instance Lefschetz-type results for global $h$-differentials; see Remark \ref{rmk:m=-1}.

Besides the consequences in this paper, these Lefschetz theorems are also used in the companion paper \cite{PP25} in order to analyze the $\Q$-factoriality defect of $X$.

\smallskip
\noindent
{\bf Rational homology manifolds.}
An obvious case where we have full symmetry of the Hodge-Du Bois diamond is when $X$ is a \emph{rational homology manifold}. According to the standard definition, this means that the homology of the link of each singularity of $X$ is the same as that of a sphere. It is known (see \cite{BM83}) that this is equivalent to the fact that the canonical  map 
$$\Q_X [n] \to \IC_X,$$
from the shifted constant sheaf to the intersection complex of $X$  is an isomorphism (which extends to the respective Hodge modules). This implies immediately that 
$$H^\bullet (X, \Q) = \mathit{IH}^\bullet (X, \Q),$$
as Hodge structures, and it is well known that intersection cohomology is well behaved.

Understanding which singularities are rational homology manifolds is a much studied topic. According to the approach in this paper, it is natural to ask how far this condition is from the symmetry of the Hodge-Du Bois diamond, as well as what either condition means in terms of standard singularity notions. 
As one of the main concrete applications of the symmetry results discussed in this paper, we answer this explicitly in low dimension. The first step is a technical result on when the two conditions coincide, Theorem \ref{thm:RHM-symmetry}, that holds in arbitrary dimension. In the Appendix, we also give a criterion for the rational homology manifold condition that holds in arbitrary dimension, in terms of local analytic cohomological invariants.

It has been known since \cite{Mumford61} that surfaces with rational singularities are rational homology manifolds. We complete the picture in 
Theorem \ref{thm:RHM-surfaces}, by showing that if $X$ is a normal surface, the converse also holds if $X$ is assumed to be Du Bois. Moreover, the rational homology manifold condition is equivalent to the full symmetry of the Hodge-Du Bois diamond when $X$ is projective.

The picture becomes more complicated in dimension three. First, we note that to have Hodge symmetry, at least under the assumption $H^2 (X, \O_X) = 0$, any Du Bois (hence any log canonical) 
projective threefold must have rational singularities; see Lemma \ref{lem:vanishing-R1}. Restricting then to the world of rational singularities, the full answer is provided by:

\begin{intro-theorem}\label{thm:RHM-threefolds}
Let $X$ be a threefold with rational singularities. Then the following are equivalent:

\noindent
(i) $X$ is a rational homology manifold. 

\noindent
(ii) The local analytic Picard groups of $X$ are torsion.

\noindent
(iii) $X$ is locally analytically $\Q$-factorial.

If in addition $X$ is projective, then they are also equivalent to:

\noindent 
(iv) The Hodge-Du Bois diamond of $X$ satisfies full symmetry, i.e.
$$
\underline h^{p,q}(X)=\underline h^{q,p}(X)=\underline h^{3-p,3-q}(X)=\underline h^{3-q,3-p}(X)
$$
for all $0\le p,q\le 3$.
\end{intro-theorem}

For terminal threefolds, the equivalence between (i) and (iii) is implicitly shown by Koll\'ar \cite[Lemma 4.2]{Kollar89}, based on results of Flenner on links of singularities.
More generally, the same equivalence is proved for klt threefolds by Grassi-Weigand in \cite[Theorem 5.8]{GW2018}, this time based on the finiteness of local fundamental groups of such singularities shown in \cite{TX17} and \cite{Braun21}.  Note that in the projective case, the symmetry of the boundary of the diamond follows from Theorem \ref{thm:Hodge-symmetry-krat}; the relationship between the remaining Hodge-Du Bois numbers in the center of the diamond and $\Q$-factoriality is one of the main topics of \cite{PP25}. It leads in particular to another proof of this theorem, that goes through the global setting even for the local statements.

A byproduct of the proof, using 
Flenner's theorem \cite{Flenner81}  describing the local analytic class group of a rational singularity, is the following general consequence in arbitrary dimension:

\begin{intro-corollary}
(i) A rational homology manifold with rational singularities is locally analytically $\Q$-factorial.

\noindent
(ii) A locally analytically $\Q$-factorial variety with rational singularities is a rational homology manifold in codimension $3$.
\end{intro-corollary}

Going back to characterizing rational homology manifolds, the picture extends to fourfolds with rational singularities, only now local analytic $\Q$-factoriality is necessary (as seen in the Corollary above), but not sufficient; a higher cohomological invariant, namely the classifying group $H^2 (U^{\rm an} \smallsetminus \{x\}, \O_{U^{\rm an} \smallsetminus \{x\}}^*)$ for analytic gerbes (i.e. gerbes with band $\O_X^*$; see e.g. \cite[Theorem 5.2.8]{Brylinski}) on a punctured analytic germ around each point comes into play.\footnote{Here it is the group itself that interests us; the gerbe interpretation is used for ease of formulation, by analogy with the fact that 
the corresponding $H^1$ group parametrizes local analytic line bundles.}

\begin{intro-theorem}\label{thm:RHM-fourfolds}
Let $X$ be a fourfold with rational and locally analytically $\Q$-factorial singularities. Then the following are equivalent:

\noindent
(i) $X$ is a rational homology manifold. 

\noindent
(ii) $X$ has torsion local analytic gerbe groups.

If in addition $X$ is projective, then they are also equivalent to: 

\noindent 
(iii) The Hodge-Du Bois diamond of $X$ satisfies full symmetry, i.e.
$$
\underline h^{p,q}(X)=\underline h^{q,p}(X)=\underline h^{4-p,4-q}(X)=\underline h^{4-q,4-p}(X)
$$
for all $0\le p,q\le 4$.
\end{intro-theorem}

Recall also that by Theorem \ref{thm:Hodge-symmetry-krat} (or Theorem \ref{cor: lcdef of k-rational}), a normal fourfold with pre-$1$-rational singularities is always a rational homology manifold.

As mentioned above, the local aspects of the criteria in Theorems \ref{thm:RHM-threefolds} and \ref{thm:RHM-fourfolds} extend with an appropriate statement to arbitrary dimension; see Theorem \ref{thm:RHM-characterization}. When $X$ is projective and satisfies the rational homology manifold condition away from finitely many points, the equivalence with the symmetry of the Hodge-Du Bois diamond can also be included, subsuming the full statements of the two theorems; see Corollary \ref{cor:RHM-general}. However, the local conditions beyond the $H^2$-cohomology groups are harder to interpret geometrically.

Finally, besides the rational homology manifold condition, in the Appendix we state a general result expressing $\lcdef (X)$ by means of local analytic cohomological invariants, phrased in terms of what we call condition $L_k$. The main results are Theorems \ref{thm:lcdef-characterization2} and \ref{thm:lcdef-characterization}, derived from two main ingredients: another interpretation in terms of links of singularities from \cite{RSW23}, and Flenner's work \cite{Flenner81}. The statements extend, over the complex numbers, the well-known result of Ogus \cite{Ogus73} and Dao-Takagi \cite{DT16} that apply for small depth, and cover the case up to $\lcdef (X) \le n-4$. As an example, here is the simplest consequence of these theorems that goes beyond this:

\begin{intro-corollary}
Let $X$ be an $n$-dimensional variety, $n \ge 5$, with rational locally analytically $\Q$-factorial singularities. Then $\lcdef (X) \le n -5$ if and only if $X$ has torsion local analytic gerbe groups. 

In particular, if $X$ is a fivefold, then $\lcdef (X) = 0$ $\iff$ $X$ has torsion local analytic gerbe groups. 
\end{intro-corollary}

Note the similarity with the hypotheses of Theorem \ref{thm:RHM-fourfolds}. A general picture emerges, that is roughly phrased as follows: the condition needed to be a rational homology manifold in dimension $n$ coincides with the condition needed to have $\lcdef (X) = 0$ in dimension $n+1$. This statement needs to be modified appropriately in general,  but is literally true when $X$ has for instance a \emph{rational isolated singularity} at $x\in X$. In this case we have:
$$X ~{\rm is~an~RHM} ~({\rm resp.} ~\lcdef (X) = 0)  \iff X ~{\rm satisfies}~L_{n-2} ~({\rm resp.}~ L_{n-3}) \textrm{ along } \{x\}.$$
See the Appendix for more general statements. Recall also that when $X$ is projective, this type of criteria tell us when Theorem \ref{thm:Lefschetz-lcdef} works just as well as in the smooth case.

\smallskip
\noindent 
{\bf Example: Threefolds with isolated singularities.}
We conclude the Introduction by exemplifying the phenomena described in this paper (and in \cite{PP25}) in the case of threefolds with isolated singularities. 
This provides a clear illustration of how, by gradually imposing more and more singularity conditions, we have more and more Hodge symmetry, up to the case of rational homology manifolds. Let $X$ be one such threefold, and consider its Hodge-Du Bois diamond:

$$
\begin{array}{ccccccc}
& & & \h^{3,3} & & & \\
&&&&&& \\
& & \h^{3,2}& &\h^{2,3} & & \\
&&&&&& \\
&  \h^{3,1} & &\h^{2,2}& & \h^{1,3} &  \\
&&&&&& \\
\h^{3,0} & & \h^{2, 1} &  & \h^{1, 2} & &\h^{0,3} \\
&&&&&& \\
&\h^{2,0}   & &\h^{1,1}& & \h^{0,2} &  \\
&&&&&& \\
& & \h^{1,0}& &\h^{0,1} & & \\
&&&&&& \\
& & & \h^{0,0} & & &
\end{array}
$$

First the already known facts: the isolated singularities condition implies  that the ``top" part of the Hodge-Du Bois diamond, i.e. $H^{p.q} (X)$ with $p + q \ge 4$, 
coincides with intersection cohomology, and therefore has horizontal symmetry $\h^{p,q} = \h^{q,p}$. If in addition $X$ is \emph{normal}, then $H^1(X, \Q)$ is pure, and therefore 
$\h^{1,0} = \h^{0,1}$; see Example \ref{ex:normal-H1}.

Now for the new input: if $X$ is \emph{Cohen-Macaulay}, then $\lcdef (X) = 0$ by a result of Dao-Takagi (see Example \ref{ex:def=0}), and therefore $H^1(X, \Q) \simeq \mathit{IH}^1 (X, \Q)$ and $H^2 (X, \Q)$ is pure, by Corollary \ref{cor:purity}. The latter implies that $\h^{2,0} = \h^{0,2}$ hence, except for $H^3 (X, \Q)$, we have horizontal symmetry.
The former implies that $H^1 (X, \Q)$ is Poincar\'e dual to $H^5 (X, \Q)$, which gives
$$\h^{1,0} = \h^{0,1} = \h^{3,2} = \h^{2,3}.$$

We need more for further symmetries. If $X$ has \emph{rational singularities}, then Theorem \ref{thm:Hodge-symmetry-krat} implies 
$$\h^{2,0} = \h^{0,2} = \h^{3,1} = \h^{1,3} \,\,\,\,{\rm and}\,\,\,\, \h^{3,0} = \h^{0,3}.$$
(The hypothesis on rational singularities is necessary, for instance when $X$ is Du Bois with $\h^{0,2} = 0$.)

Hence under the rational singularities assumption, the only potential source of non-symmetry is the middle rhombus (even when the singularities are not isolated). 
Theorem \ref{thm:RHM-threefolds} states that its symmetry holds when $X$ is in addition \emph{locally analytically $\Q$-factorial}, and that this is in fact equivalent to $X$ being a rational homology manifold.  A more precise answer is provided in \cite{PP25}. First, we have that $\h^{1,1} = \h^{2,2}$ if and only if $X$ is \emph{$\Q$-factorial}. Finally, we have that in addition 
$\h^{2,1} = \h^{1,2}$ if and only if $X$ is locally analytically $\Q$-factorial.

\smallskip
\noindent
{\bf Acknowledgements.} 
We thank Bhargav Bhatt, Brad Dirks, Robert Friedman, Antonella Grassi, J\'anos Koll\'ar, Radu Laza, Lauren\c tiu Maxim, Rosie Shen, and Duc Vo for valuable discussions.

\section{Preliminaries}

\subsection{Du Bois complexes and higher singularities}\label{scn:DB}
For a complex variety $X$, the \emph{filtered de Rham complex} $(\DB_X^\bullet, F)$ is an object in the bounded derived category of filtered differential complexes on $X$, suggested by Deligne and introduced by Du Bois in \cite{DB81}  as a replacement for the standard de Rham complex on smooth varieties. For each $k \ge 0$, the shifted associated graded quotient
\[\DB^k_X : = \gr^k_F\, \DB^\bullet_X[k],\]
is an object in ${\bf D}^b_{\rm coh}(X)$, called the \emph{$k$-th Du Bois complex} of $X$. 
It can be computed in terms of the de Rham complexes on a hyperresolution of $X$.
Besides \cite{DB81}, one can find a detailed treatment in \cite[Chapter V]{GNPP} or \cite[Chapter 7.3]{PS08}.

Here are a few basic properties of Du Bois complexes that will be used throughout the paper:

\smallskip
\noindent
	$\bullet$ ~~For each $k\ge 0$, there is a canonical morphism $\Omega_X^k \to \DB_X^k$, which is an isomorphism if $X$ is smooth; here 
	$\Omega_X^k$ are the sheaves of K\"ahler differentials on $X$; see \cite[Section 4.1]{DB81} or \cite[Page 175]{PS08}.

\smallskip
\noindent
$\bullet$~~For each $k\ge 0$, the sheaf $\H^0\DB_X^k$ embeds in $f_* \Omega_{\widetilde{X}}^k$, where $f\colon \widetilde{X} \to X$ is a resolution of $X$, 
so in particular it is torsion-free; see \cite[Remark 3.8]{HJ14}.

\smallskip		
\noindent
$\bullet$~~	There exists a Hodge-to-de Rham spectral sequence 
	$$E^{p,q}_1 = \HH^q (X, \DB_X^p) \implies H^{p + q} (X, \C),$$
	which degenerates at $E_1$ if $X$ is projective; see \cite[Theorem 4.5(iii)]{DB81} or \cite[Proposition 7.24]{PS08}. The filtration on each $H^i (X, \C)$ given by this spectral sequence coincides with the Hodge filtration associated to Deligne's mixed Hodge structure on cohomology.

\medskip
The Du  Bois complexes can be used to produce a hierarchy of singularities, starting with the standard notions of Du Bois and rational singularities.
Following \cite{MOPW,JKSY,FL22, MuP22b}, if $X$ is a local complete intersection (LCI) subvariety of a smooth variety $Y$, then it is said to have \textit{$m$-Du Bois singularities} if the canonical morphisms $\Omega^p_X \to \DB^p_X$ are isomorphisms for all $0\le p\le m$, and 
 \textit{$m$-rational singularities} if the canonical morphisms $\Omega^p_X \to \DD_X(\DB^{n-p}_X)$  are isomorphisms for all $0\le p\le m$, where 
 $\DD_X (\cdot) : = R\H om (\, \cdot \, , \omega_X)$.
 
For non-LCI varieties these conditions are too restrictive, as explained in \cite{SVV23}, where new definitions are introduced in the general setting.
However, it is often sufficient to consider weaker notions obtained by removing the conditions in cohomological degree 0.

\begin{defn}\label{definition:pre-k-DB-rational}
	One says  that $X$ has \textit{pre-$m$-Du Bois} singularities if 
		\[ \H^i\DB_X^p=0 \,\,\,\,{\rm for~ all } \,\,\,\,i>0 \,\,\,\, {\rm and} \,\,\,\, 0\le p\le m,\]
and that it has \textit{pre-$m$-rational} singularities if
		\[ \H^i (\DD_X(\DB^{n-p}_X))=0 \,\,\,\,{\rm for~ all } \,\,\,\,i>0 \,\,\,\, {\rm and} \,\,\,\, 0\le p\le m.\]
\end{defn}

General $m$-Du Bois and $m$-rational singularities require further conditions; see \cite{SVV23}. They agree with the classical notions when $m=0$, and  with the definitions mentioned above in the local complete intersection case.

\iffalse

; see \cite[Proposition 5.5, 5.6]{SVV23} for more details.

\begin{defn} 
	We say that $X$ has \textit{$k$-Du Bois singularities} if it is seminormal, and 
	\begin{enumerate}
		\item $\codim_X \Sing(X) \ge 2k+1$;
		\item $X$ has pre-$k$-Du Bois singularities;
		\item $\H^0\DB_X^p$ is reflexive, for all $p \le k$.
	\end{enumerate}
	We say that $X$ has \textit{$m$-rational singularities} if it is normal, and
	\begin{enumerate}
		\item $\codim_X \Sing(X) > 2k+1$;
		\item $X$ has pre-$m$-rational singularities. 
	\end{enumerate}
\end{defn}

\fi

Later we will discuss the relationship between these types of singularities and intersection complexes, which in turn will reveal a precise connection with the topology of the manifold. One significant consequence, see Theorem \ref{thm:Hodge-symmetry-krat}, will be that a variety with normal pre-$m$-rational singularities such that $m \ge (n-2)/2$ is a rational homology manifold.

\subsection{Local cohomological defect}\label{scn:lcdef}
Let $X$ be an equidimensional variety of dimension $n$.
If $Y$ is a smooth variety containing $X$ (locally), the local cohomological dimension of $X$ in $Y$ is
$${\rm lcd} (X, Y) := \max ~\{q ~|~ \H^q_X \O_Y \neq 0\}.$$
Here $\H^q_X \O_Y$  is the $q$-th local cohomology sheaf of $\O_Y$ along $X$. These sheaves also satisfy the property 
$$\codim_Y X := \min ~\{q ~|~ \H^q_X \O_Y \neq 0\}.$$
A discussion of this circle of ideas can be found for instance in \cite[Section 2.2]{MuP22a}, where it is explained in particular that 
\begin{equation}\label{eqn:lcd-equations}
{\rm lcd} (X, Y) \le r (X, Y),
\end{equation}
where $r (X, Y)$ is the minimal number of defining equations for $X$ in $Y$.

As in \cite{PS24}, \cite{PSV24}, we consider the \emph{local cohomological defect} ${\rm lcdef} (X)$ of $X$ as
$${\rm lcdef} (X) : = {\rm lcd} (X, Y) - \codim_Y X.$$
The topological characterization of ${\rm lcd} (X, Y)$ in \cite{Ogus73}, or its holomorphic characterization in \cite{MuP22a}, imply that $\lcdef (X)$ 
depends only on $X$, and not on the embedding, and that 
 $0 \le  {\rm lcdef} (X) \le n$. 
 
For instance, as explained in \cite[Section 2]{PS24}, a reinterpretation of the characterization of local cohomological dimension in \cite[Theorem E]{MuP22a}, that makes the relationship with Hodge theory apparent, is stated as follows:

\begin{thm}\label{cor:depth-DB-complex}
We have the identity 
$${\rm lcdef} (X) = n -  \underset{k \ge 0}{\min}~ \{ {\rm depth}~\DB_X^k + k \}.$$
\end{thm}

Here we will be more often concerned with the topological interpretation of ${\rm lcdef} (X)$, more precisely with its more modern interpretation via Riemann-Hilbert correspondence and the perverse $t$-structure on the bounded derived category $D^b_c(X,\Q)$ of $\Q$-sheaves on $X$ with constructible cohomologies, as formulated in \cite{RSW23} and  \cite{BBLSZ}.

\begin{thm}[{\cite[Theorem 2.13]{Ogus73}, \cite[Theorem 1]{RSW23}, \cite[Section 3.1]{BBLSZ}}]\label{thm:lcdef-top}
We have the identity 
$${\rm lcdef} (X) = {\rm max}~\{j \ge 0 ~|~ {}^p\H^{-j} (\Q_X [n]) \neq 0\}.\footnote{We include the shift for compatibility with the theory of perverse sheaves and Hodge modules, used later on.}$$
\end{thm}

\medskip

Note that thanks to ($\ref{eqn:lcd-equations}$) we have
 \begin{equation}\label{eqn:lcd-equations-2}
 {\rm lcdef} (X) \le s (X) : = r (X, Y) - \codim_Y X.
 \end{equation}
 The quantity on the right hand side measures the defect of being a local complete intersection. When $X$ is so, we obviously have ${\rm lcdef}(X)= 0$. We will see 
that in many respects the topology of $X$ is best behaved when this condition is satisfied, or in any case when $\lcdef (X)$ is small. Interestingly, this can happen for varieties 
 that may be quite far from being local complete intersections.

\begin{ex}[{\bf Varieties with ${\rm lcdef}(X)= 0$}]\label{ex:def=0}
Besides local complete intersections, the condition ${\rm lcdef} (X) = 0$ it is known to hold 
when $X$ has quotient singularities (see \cite[Corollary 11.22]{MuP22a}, or use the known fact that $X$ is a rational homology manifold), for affine varieties with Stanley-Reisner coordinate algebras that are Cohen-Macaulay \cite[Corollary 11.26]{MuP22a}, and for arbitrary Cohen-Macaulay surfaces and threefolds in \cite[Remark p.338-339]{Ogus73} and \cite[Corollary 2.8]{DT16} respectively.

More generally, Dao-Takagi \cite[Theorem 1.2]{DT16} show that for $k \le 3$ one has 
$$
\mathrm{depth} ~\O_X \ge k \Longrightarrow \lcdef(X)\le n-k  
$$
This does not continue to hold for $k \ge 4$, but for instance it implies that a Cohen-Macaulay fourfold satisfies $\lcdef (X) \le 1$. Moreover, in this case we do have 
$\lcdef (X) = 0$ if the local analytic Picard groups are torsion \cite[Theorem 1.3]{DT16}.  

Section \ref{scn:appendix} is devoted to further similar criteria in higher dimension.
\end{ex}

\subsection{Intersection complex, rational homology manifolds, condition $D_m$}\label{scn:IC}
Let $X$ be a complex variety of dimension $n$. The trivial Hodge module $\Q^H_X[n]$ is an object in the derived category of mixed Hodge modules on $X$, with cohomologies in degrees $\le 0$. We also have the intersection complex ${\rm IC}_X^H$, a simple pure Hodge module of weight $n$.
The top degree of the weight filtration on $\H^0  \Q^H_X[n]$ is $n$, and there is a composition of quotient morphisms
\begin{equation}\label{eqn:gamma}
\gamma_X: \Q^H_X[n]\to  \H^0  \Q^H_X[n]\to \IC_X^H \simeq \gr^W_n  \H^0  \Q^H_X[n].
\end{equation}
For details on all of this, please see \cite[Section 4.5]{Saito90}.

Thanks to Theorem \ref{thm:lcdef-top}, we know that the degree of the lowest nontrivial cohomology of $\Q^H_X[n]$ is equal to $- \lcdef (X)$, i.e.
\begin{equation}\label{eqn:lcdef-MHM}
\lcdef (X) = {\rm max}~ \{j \ge 0 ~|~ \H^{-j} (\Q^H_X[n]) \neq 0\}.
\end{equation}

\medskip

Saito \cite[Theorem 0.2]{Saito00} proved that the Du Bois complexes of $X$ can be identified up to shift with the graded pieces of the de Rham complex of $\Q^H_X [n]$, namely
$$
\DB^k_X\isom\gr^F_{-k}\DR(\Q^H_X[n])[k-n].
$$
We consider the formal analogue in the case of the intersection complex.

\begin{defn}
The \textit{intersection Du Bois complex} $I\DB^k_X$ of an equidimensional variety $X$ of dimension $n$ is the graded de Rham complex of the Hodge module $\IC^H_X$ shifted by $[k-n]$:
$$
I\DB^k_X := \gr^F_{-k}\DR(\IC_X^H)[k-n].
$$
\end{defn}

We note an important property of $I\DB_X^k$, used in \cite{KS21} to prove results on the extension of holomorphic forms (see also \cite{Park23}). 

\begin{lem}[{\cite[(8.4.1) and Proposition 8.1]{KS21}}]
\label{prop: 0th cohomology of IDB}
Let $X$ be an equidimensional variety of dimension $n$. Then for any resolution of singularities $f:\widetilde X\to X$, we have that  $I\DB^k_X$ is a direct summand 
of $R f_*\Omega^k_{\widetilde X}$, so in particular its cohomologies are supported in non-negative degrees. Moreover, there is an isomorphism
$$
f_*\Omega^k_{\widetilde X}\isom \H^0I\DB^k_X.
$$
\end{lem}

Since any associated graded of the de Rham complex of a Hodge module with respect to the Hodge filtration lives in nonpositive degrees, it is clear by definition 
that we also have 
$$\H^i I\DB_X^k = 0 \,\,\,\,\,\,{\rm for }\,\,\,\, i > n -k.$$
Several other properties of these complexes are studied in \cite{PSV24}.

\medskip

A \emph{rational homology manifold (or RHM)} is classically defined by the fact that the homology of the link of each singularity of $X$ is the same as that of a sphere. 
It is however known (see \cite{BM83}) that this is equivalent to the fact that the map $\gamma_X\colon \Q^H_X[n]\to  \IC_X^H$ is an isomorphism. We thus have the implication
$$X {\rm~ is~an~RHM} \implies \lcdef (X) = 0.$$
This is usually not an equivalence. For instance, any normal surface has $\lcdef (X) =0$; however, when $X$ is Du Bois, the RHM condition is equivalent to $X$ having rational 
singularities (see Theorem \ref{thm:RHM-surfaces}).

Applying the functor $\gr^F_{-k} \DR$ to $\gamma_X$ in ($\ref{eqn:gamma}$) yields morphisms 
\begin{equation}\label{DB-to-IC}
\gamma_k: \DB^k_X \to I\DB^k_X.
\end{equation}

It is straightforward to see the following:

\begin{lem}\label{lem:RHM-gamma}
The variety $X$ is a rational homology manifold if and only if $\gamma_k$ is an isomorphism for all $k$. 
\end{lem}

A crucial point will be to study what happens when the $\gamma_k$ are isomorphisms only in a certain range. It will be convenient to introduce notation that keeps track of this, and will be used repeatedly throughout the paper.

\begin{defn}[\bf{Condition $D_m$}]
We say that $X$ satisfies condition $D_m$ if the morphisms $\gamma_p: \DB^p_X \to I\DB^p_X$ are isomorphisms for all $0 \le p \le m$. For simplicity, we will sometimes write $\DB^p_X = I\DB^p_X$ to mean this. 
\end{defn}

\begin{rmk}[Notation]\label{rmk:terminology}
The reason for this notation is the following: it is quite well known by now that the condition in the definition above is equivalent to requiring that the natural 
duality morphisms 
$$\DB_X^p \to \DD_X(\DB^{n-p}_X)$$ 
are isomorphisms for $0 \le p \le m$; see e.g. the proof of \cite[Proposition 9.4]{PSV24}. Thus $D_m$, which is perhaps more natural in the context of the latter condition, stands for \emph{duality up to level $m$}. In the previous version of this paper \cite{PP} we called this condition $(*)_m$, but the current notation seems more suggestive. Note that condition $D_m$ is also studied, in the latter form, and in terms of an invariant denoted ${\rm HRH} (X)$, in the recent \cite{DOR}.

In particular, note that $X$ has $m$-rational singularities if and only if it has $m$-Du Bois singularities and satisfies condition $D_m$, and 
having $m$-rational singularities is typically strictly stronger than condition $D_m$. As a toy illustration, a variety with rational singularities satisfies $\O_X \simeq \DB_X^0 \simeq I \DB_X^0 \simeq R f_* \O_{\widetilde X}$, while condition $D_0$ only says that $\DB_X^0 \simeq 
I \DB_X^0 \simeq Rf_* \O_{\widetilde X}$, where $f \colon \widetilde X \to X$ is a resolution. Moreover,  we have 

\begin{prop}[{\cite[Proposition 9.4]{PSV24}}]\label{prop:krat*}
If $X$ is normal, with pre-$m$-rational singularities, then $X$ satisfies condition $D_m$.
\end{prop}

In fact, the converse is also known under the pre-$m$-Du Bois assumption; see \cite[Sections 4, 5]{DOR}, and earlier \cite[Theorem 3.1]{CDM22} for local complete intersections.
\end{rmk}

\begin{rmk}[Defect]
Given Lemma \ref{lem:RHM-gamma}, for the largest $m$ such that $X$ satisfies condition  $D_m$, it may be tempting to think of $n-m$ as a  \emph{rational homology manifold defect} of $X$. Proposition \ref{prop: 0 to k then n-k-1 to n} will show however that such a defect is more accurately described by the quantity $\lceil \frac{n-2}{2} \rceil - m$.
\end{rmk}

For many topological and Hodge theoretic arguments, the key assumption is $D_m$, as we will see throughout the paper. The stronger rationality conditions are related to the finer analytic properties of the singularities.

\subsection{Hodge-Du Bois numbers and intersection Hodge numbers}\label{scn:HDB}
Let $X$ be a projective variety of dimension $n$. The ordinary cohomology $H^{p+q}(X,\Q)$ has a mixed Hodge structure with (increasing) Hodge filtration $F_\bullet$, and as for smooth varieties, the Hodge numbers $\underline h^{p,q}$ are defined by
$$
\underline h^{p,q}(X):=\dim_\C \gr^F_{-p}H^{p+q}(X,\C).
$$
As mentioned in Section \ref{scn:DB}, these are also computed by the hypercohomologies of the Du Bois complexes $\DB_X^p$, in the sense that 
$$
H^{p, q} (X): = \HH^q(X,\DB_X^p)=\gr^F_{-p}H^{p+q}(X,\C).
$$
Therefore, to emphasize the singular setting and the difference with the Hodge-Deligne numbers of the pure Hodge structures on the associated graded pieces with respect to the 
weight filtration, one sometimes refers to them as \emph{Hodge-Du Bois numbers}. As in the smooth setting, they organize themselves into a \emph{Hodge-Du Bois diamond}:
$$
\begin{array}{ccccccc}
& & & \h^{n,n} & & & \\
&&&&&& \\
& & \h^{n,n-1}& &\h^{n-1,n} & & \\
&  \Ddots & &\vdots & & \ddots &  \\
&&&&&& \\
\h^{n,0} & \h^{n-1, 1} & & \ldots & & \h^{1, n-1} &h^{0,n} \\
&&&&&& \\
& \ddots   & &\vdots & & \Ddots &  \\
& & \h^{1,0}& &\h^{0,1} & & \\
&&&&&& \\
& & & \h^{0,0} & & &
\end{array}
$$
To understand the symmetries of this diamond, one of the key tools is the comparison with intersection cohomology, by means of the 
intersection complex Hodge module $\IC^H_X$ and the graded pieces of its filtered de Rham complex, considered in the previous section. 
We denote the intersection cohomology of $X$ by $\mathit{IH}^\bullet(X)$. We have 
$$\mathit{IH}^{i + n} (X, \C) = \HH^i (X, \IC_X),$$
and by Saito's theory \cite{Saito88, Saito90}, this carries a pure Hodge structure. We set
$$
\mathit{IH}^{p,q}(X) : = \gr^F_{-p}\mathit{IH}^{p+q}(X,\C)=  \HH^q(X,I\DB_X^p)
$$
and  the \textit{intersection Hodge numbers} of $X$ are defined as 
$$I\underline h^{p,q}(X)  := \dim_\C \mathit{IH}^{p,q}(X).$$
Intersection cohomology also satisfies Poincar\'e duality, so altogether it has the full Hodge symmetry
$$
I\underline h^{p,q}(X)=I\underline h^{q,p}(X)=I\underline h^{n-p,n-q}(X)=I\underline h^{n-q,n-p}(X).
$$
Applying $\HH^q (X, \, \cdot \,)$ to the morphism $\gamma_p$ in the previous section induces the natural map 
$$H^{p, q} (X)  \to \mathit{IH}^{p,q}(X),$$
which in turn is a Hodge piece of the natural topological map 
$$H^{p+q} (X, \C) \to \mathit{IH}^{p+q} (X, \C).$$
The more agreement there is between singular cohomology and intersection cohomology via these mappings, the more symmetry we have in the Hodge-Du Bois diamond.

\begin{lem}\label{lem:*k-coincidence}
If $X$ satisfies condition $D_m$, we have  
$$ H^{p,q} (X) = \mathit{IH}^{p,q} (X) \,\,\,\,\,\, {\rm for~ all} \,\,\,\, 0 \le p \le m {\rm ~and~} 0\le q\le n.$$
\end{lem}

This culminates in the case of rational homology manifolds (see Lemma \ref{lem:RHM-gamma}), when the two agree completely.

\medskip
It is also worth noting that simply due to the fact that they are the (total) Hodge numbers of the Hodge filtration on a mixed Hodge structure, the Hodge-Du Bois numbers satisfy some universal numerical inequalities, as noted by Friedman and Laza.

\begin{lem}[{\cite[Lemma 3.23]{FL22}}]\label{lem:FL}
Let $X$ be a complex projective variety of dimension $n$. For all $0 \le p \le i \le n$ we have
$$\sum_{a=0}^p \h^{i - a, a} (X) \le \sum_{a=0}^p \h^{a, i - a} (X).$$
Moreover, equality holds for all $0 \le p \le m$ $\iff \h^{i - p, p} (X) = \h^{p, i - p} (X)$ for all ~$0 \le p \le m$ $\iff \gr_F^p W_{i -1} H^i (X, \C) = 0$ for 
all $p \le m$.
\end{lem}

This contains the well-known fact that the Hodge structure on $H^i (X, \Q)$ is pure if and only if $\h^{i - p, p} (X) = \h^{p, i - p} (X)$ for all ~$0 \le p \le i$, which will be used repeatedly throughout the paper.

\subsection{Some generalities on mixed Hodge modules}
We begin by stating a general lemma which shows the existence of a morphism from an object in the derived category of mixed Hodge modules to its top weight graded complex. This will be used to control the weight of the key object considered in this paper, the ${\rm RHM}$-defect defined in the next section.

\begin{lem}
\label{lem: map to top weight graded}
Let $\M^\bullet\in D^b {\rm MHM}(X)$ be an object of weight $\le m$ in the derived category of mixed Hodge modules on $X$, in other words
$$
\gr^W_i\H^j(\M^\bullet)=0 \quad \mathrm{for} \quad i>j+m.
$$
Then there exists a morphism
$$
\M^\bullet\to \bigoplus_i\gr^W_{m+i}\H^{i}(\M^\bullet)[-i]
$$
in $D^b{\rm MHM}(X)$ such that the induced morphism of cohomologies
$$
\H^i(\M^\bullet)\to\gr^W_{m+i}\H^{i}(\M^\bullet)
$$
is the natural quotient morphism obtained from the fact that $W_{m+i}\H^i(\M^\bullet)=\H^i(\M^\bullet)$.
\end{lem}

\begin{proof}
From basic properties of mixed Hodge modules, we have
$$
\mathrm{Hom}_{D^b {\rm MHM}(X)}(\M^\bullet,\N^\bullet)=0
$$
when $\M^\bullet$ is of weight $\le w$ and $\N^\bullet$ is of weight $>w$ for some integer $w$ (see \cite[Lemma 14.3]{PS08} or \cite[4.5.3]{Saito90}). Using this, we proceed by induction on the number of elements in the set
$$
S(\M^\bullet):=\left\{i\;|\;\gr^W_{m+i}\H^{i}(\M^\bullet)\neq 0\right\}.
$$

When $|S(\M^\bullet)|=0$, the lemma is immediate. Suppose $|S(\M^\bullet)|$ is positive. Let $k$ be the largest integer in $S(\M^\bullet)$. Consider the following distinguished triangle
$$
\tau_{\le k}\M^\bullet\to \M^\bullet\to \mathcal{Q}^\bullet\xrightarrow{+1}
$$
where $\tau_{\le k}\M$ is a canonical truncation whose cohomologies are supported on degrees $\le k$. Therefore, there exists a long exact sequence
\begin{multline*}
\to \mathrm{Hom}(\M^\bullet, \gr^W_{m+k}\H^{k}(\M^\bullet)[-k])\to \mathrm{Hom}(\tau_{\le k}\M^\bullet, \gr^W_{m+k}\H^{k}(\M^\bullet)[-k]) \\
\to \mathrm{Hom}(\mathcal Q^\bullet[-1], \gr^W_{m+k}\H^{k}(\M^\bullet)[-k])\to 
\end{multline*}
Since $\mathcal Q^\bullet$ is of weight $\le m-1$, its shift $\mathcal Q^\bullet[-1]$ is of weight $\le m-2$. Hence,
$$
\mathrm{Hom}(\mathcal Q^\bullet[-1], \gr^W_{m+k}\H^{k}(\M^\bullet)[-k])=0
$$
and the natural surjection $\tau_{\le k}\M^\bullet\to \gr^W_{m+k}\H^{k}(\M^\bullet)[-k]$ lifts to a morphism
$$
\M^\bullet\to\gr^W_{m+k}\H^{k}(\M^\bullet)[-k].
$$

We extend this morphism to a distinguished triangle
$$
\mathcal R^\bullet\to\M^\bullet\to\gr^W_{m+k}\H^{k}(\M^\bullet)[-k]\xrightarrow{+1}.
$$
Then, $S(\M^\bullet)=S(\mathcal R^\bullet)\cup \left\{k\right\}$. By induction hypothesis, there exists a morphism
$$
\mathcal R^\bullet\to \bigoplus_i\gr^W_{m+i}\H^{i}(\mathcal R^\bullet)[-i]=\bigoplus_{i\neq k}\gr^W_{m+i}\H^{i}(\mathcal M^\bullet)[-i]
$$
and this morphism lifts to a morphism
$$
\M^\bullet\to \bigoplus_{i\neq k}\gr^W_{m+i}\H^{i}(\mathcal M^\bullet)[-i],
$$
because $\mathrm{Hom}(\gr^W_{m+k}\H^{k}(\M^\bullet)[-k-1],\bigoplus_i\gr^W_{m+i}\H^{i}(\mathcal R^\bullet)[-i])=0$. Indeed, $\gr^W_{m+k}\H^{k}(\M^\bullet)[-k-1]$ has weight $m-1$ and $\bigoplus_i\gr^W_{m+i}\H^{i}(\mathcal R^\bullet)[-i]$ has weight $m$. Combining two morphisms, we obtain a morphism
$$
\M^\bullet\to \bigoplus_{i}\gr^W_{m+i}\H^{i}(\mathcal M^\bullet)[-i]
$$
in $D^b{\rm MHM}(X)$, and this completes the proof.
\end{proof}

\medskip

In a different direction, given a polarizable $\Q$-Hodge structure $(V,F_\bullet)$ of weight $m$, we have a Hodge decomposition
$$
V_\C:=V\tensor_\Q \C=\bigoplus_{p+q=m}V^{p,q},\quad V^{p,q}\isom \gr^F_{-p}V_\C,
$$
such that $V^{p,q}=\overline{V^{q,p}}$. (Recall that we are using the increasing convention for the Hodge filtration.)
In particular, this implies that if $V^{p,q}=0$ for $p\le k$, then $V^{p,q}=0$ for $p\ge m-k$. In other words, if $\gr^F_{-p}V=0$ for $p\le k$, then this also holds for $p\ge m-k$. It is easy to see that this remains true for mixed $\Q$-Hodge structures of weight $\le m$, and this generalizes to mixed Hodge modules in the following sense.

\begin{prop}
\label{prop: range of nonzero grdr}
Let $\M\in {\rm MHM}(X)$ be a mixed Hodge module on $X$ of weight $\le m$. If 
$$
\gr^F_{-p}\DR(\M)=0\,\,\,\,\,\,{\rm for} \,\,\,\,p \le k,
$$
then this also holds for $p\ge m-k$.
\end{prop}

\begin{proof}
We prove this by induction on the dimension of $\Supp(\M)$.

The initial case, when $\M$ is a mixed $\Q$-Hodge structure of weight $\le m$ supported at a point, is explained in the previous paragraph. Suppose $\dim(\Supp(\M))\ge 1$.

Since the statement is local, we assume that $X$ is embeddable in a smooth quasi-projective variety $Y$. For a very general hyperplane section $\iota:L\hookrightarrow Y$, we denote $\M_L:=\H^{-1}(\iota^*\M)$. Then $\M_L$ is of weight $\le m-1$ (see \cite[4.5.2]{Saito90}). According to Lemma \ref{lem:hyperplane-de Rham} below, we have a distinguished triangle
$$
N^*_{L/Y}\otimes_{\O_L}\gr^F_{-p+1}\DR\left(\M_L\right)\to \O_L\tensor_{\O_Y}\gr^F_{-p}\DR\left(\M\right)\to \gr^F_{-p}\DR\left(\M_L\right)[1]\xrightarrow{+1}.
$$
Since $\gr^F_{-p}\DR(\M_L)=0$ for $p\ll0$, it is easy to see that $\gr^F_{-p}\DR(\M_L)=0$ for $p\le k$ by an inductive argument.

By induction hypothesis, we have $\gr^F_{-p}\DR(\M_L)=0$ for $p\ge m-1-k$. Therefore, the distinguished triangle implies that
$$
\O_L\tensor_{\O_Y}\gr^F_{-p}\DR\left(\M\right)=0
$$
for $p\ge m-k$. Hence, the cohomologies of $\gr^F_{-p}\DR\left(\M\right)$ are supported at closed points of $X$ for $p\ge m-k$.

Let $p_0$ be a largest integer such that $\gr^F_{-p_0}\DR\left(\M\right)$ is nonzero. In other words, $-p_0$ is the index of the first nonzero term in the Hodge filtration of $\M$. It suffices to show that $p_0\le m-k-1$. We argue by contradiction that $p_0\ge m-k$. If $(M,F_\bullet M, W_\bullet M)$ is the underlying filtered $\D_Y$-module of $\M$, then $F_{-p_0-1}M=0$ and
$$
\gr^F_{-p_0}\DR\left(\M\right)\isom F_{-p_0}M.
$$
See \cite[Section 3]{Park23} for details. By the compatibility and strictness of Hodge filtration $F_\bullet$ and weight filtration $W_\bullet$, \cite[Lemme 5]{Saito88}, there exists an integer $w\le m$ such that
$$
F_{-p_0}\gr^W_{w}M\neq 0.
$$

Since the category of pure Hodge modules is semisimple, \cite[Lemme 5]{Saito88}, $\gr^W_w \M$ decomposes into irreducible Hodge modules
$$
\gr^W_w \M\isom \bigoplus_{i}\M_i.
$$
Hence, there exists one component $\M_0$ with the underlying filtered $D_Y$-module $(M_0,F_\bullet M_0)$ such that $F_{-p_0}M_0\neq0$, and $F_{-p_0}M_0$ is supported at a closed point $x$ of $X$. From the basic property of irreducible Hodge modules, $\M_0$ is a Hodge module corresponding to the minimal extension of an irreducible polarizable variation of $\Q$-Hodge structures supported on a subvariety of $X$. Therefore, $\M_0$ is a Hodge module corresponding to a polarizable $\Q$-Hodge structure $(V,F_\bullet)$ of weight $w$, supported at $x\in X$, with $F_{-p_0}V_\C\neq 0$. This implies that $\gr^F_{-w+p_0}V_\C\neq 0$.

Passing to duals, $\dual \M_0$ is a subquotient of $\dual \M$, corresponding to a $\Q$-Hodge structure $\dual V$ supported at a point $x\in X$, with $\gr^F_{w-p_0}(\dual V)_\C\neq 0$. Again, by the compatibility and strictness of Hodge filtration and weight filtration, the index of the first nonzero term in the Hodge filtration of $\dual\M$ is at most $w-p_0$. On the other hand, we have
$$
\gr^F_{p'}\DR(\dual \M)=0 \quad \mathrm{for}\quad p'\le k
$$
from the duality formula (see e.g. \cite[Lemma 3.2]{Park23}). Therefore, we have $k< w-p_0$, and thus $p_0\le m-k-1$. This completes the proof.
\end{proof}

Finally, for repeated use, we recall a useful lemma about the restriction of the graded quotients of the de Rham complex to a very general hypersurface; see \cite[Proposition 4.17]{KS21} and \cite[13.3]{Schnell16}.

\begin{lem}\label{lem:hyperplane-de Rham}
Let $Y$ be a quasi-projective variety, and $\iota:L\hookrightarrow Y$ a very general hyperplane section. For a mixed Hodge module $\M\in {\rm MHM}(Y)$, we have a short exact sequence of complexes
$$
0\to N^*_{L/Y}\otimes_{\O_L}\gr^F_{p+1}\DR\left(\M_L\right)\to \O_L\tensor_{\O_Y}\gr^F_{p}\DR\left(\M\right)\to \gr^F_{p}\DR\left(\M_L\right)[1]\to 0
$$
where $\M_L=\H^{-1}(\iota^*\M)$.
\end{lem}

\section{Hodge-Du Bois symmetry I}

\subsection{The RHM defect object}\label{scn:RHM}
We introduce now the crucial ingredient we use in order to compare singular cohomology and intersection cohomology. Throughout this section $X$ is  
an equidimensional variety of dimension $n$.

\begin{defn}
The \emph{RHM-defect object of $X$} is the object $\K_X^\bullet\in D^b{\rm MHM}(X)$ sitting in the distinguished triangle:
\begin{equation}
\label{eqn: Q to IC triangle}
\K_X^\bullet \longrightarrow \Q_X^H[n]\overset{\gamma_X}{\longrightarrow} \IC_X^H\xrightarrow{+1}.    
\end{equation}
\end{defn}

Thus $X$ is a rational homology manifold if and only if $\K_X^\bullet = 0$. Moreover, by ($\ref{eqn:lcdef-MHM}$),
$$\lcdef (X) = 0 \iff \H^j (\K^\bullet_X)=0 \,\,\,\,{\rm for} \,\,\,\,j \neq 0.$$

\begin{rmk}
In this latter case, or equivalently when $\Q_X[n]$ is perverse,  the $\Q$-perverse sheaf underlying $\K_X^\bullet$ has already appeared as the ``multiple-point complex" in \cite{HM16}. It was later studied as the ``comparison complex", for instance in \cite{Hepler19} and \cite{Massey22}, primarily to investigate the topology of hypersurfaces.
\end{rmk}

Some fundamental properties of $\K_X^\bullet$ are summarized in the following:

\begin{prop}
\label{prop: weight of K}
For all integers $i$, we have
$$
\dim\left(\Supp\left(\H^i\K^\bullet_X\right)\right)\le n+i-1.
$$
Additionally, $\K^\bullet_X$ is of weight $\le n-1$, in other words
$$
\gr^W_i\H^j(\K^\bullet_X)=0 \quad \mathrm{for} \quad i>j+n-1.
$$
\end{prop}

Here, the support of a mixed Hodge module is the support of the underlying D-module or perverse sheaf.

\begin{proof}
 When $n=0$, that is $X$ is a union of closed points, then clearly $\K^\bullet_X=0$. We assume $n\ge 1$.

Since the statements are local, we may assume that $X$ is embeddable in a smooth variety.
From the local Lichtenbaum theorem (\cite[Theorem 3.1]{Hartshorne68} or \cite[Corollary 2.10]{Ogus73}), we have that $\mathrm{lcd}(X,Y)\le \dim Y-1$ for any closed embedding $X\hookrightarrow Y$ into a smooth variety (see also \cite[Corollary 11.9]{MuP22a}). By the Riemann-Hilbert correspondence (see e.g. \cite[Theorem 1]{RSW23}) this translates into
$$
\H^{-n}\left(\Q^H_X[n]\right)=0.
$$
This implies that $\H^{-n}\K^\bullet_X=0$. By Lemma \ref{lem: cutting hyperplane of K} below, on the behavior of $\K^\bullet_X$ with respect to taking very general hyperplane sections, it is then straightforward to conclude by induction that
$$
\dim\left(\Supp\left(\H^i\K^\bullet_X\right)\right)\le n+i-1.
$$

It remains to prove that $\K^\bullet_X$ is of weight $\le n-1$. From \cite[4.5.2]{Saito90} we know that the object $\Q^H_X[n]\in D^b{\rm MHM}(X)$ is of weight $\le n$, 
hence so is $\K^\bullet_X$. By Lemma \ref{lem: map to top weight graded}, we have a morphism
$$
\Q^H_X[n]\to \bigoplus_i\gr^W_{n+i}\H^{i}(\Q^H_X[n])[-i].
$$
If $i<0$, then $\gr^W_{n+i}\H^{i}(\Q^H_X[n])=\gr^W_{n+i}\H^{i}(\K^\bullet_X)$ is a pure Hodge module supported on a subvariety $Z$ of dimension $\le n+i-1$ (as shown above). If $\iota:Z\hookrightarrow X$ denotes the closed embedding, by adjunction we have
$$
\mathrm{Hom}_{D^b{\rm MHM}(X)}(\Q^H_X[n], \gr^W_{n+i}\H^{i}(\Q^H_X[n])[-i])=\mathrm{Hom}_{D^b{\rm MHM}(X)}(\iota_*\Q^H_Z[n], \gr^W_{n+i}\H^{i}(\Q^H_X[n])[-i]).
$$
Since $Z$ is of dimension $\le n+i-1$, we obtain $\iota_*\Q^H_Z[n]\in D^{\le i-1}{\rm MHM}(X)$, meaning its cohomologies are supported on degrees $\le i-1$. On the other hand, $\gr^W_{n+i}\H^{i}(\Q^H_X[n])[-i]$ is an object whose only nonzero cohomology is in degree $i$. Therefore,
$$
\mathrm{Hom}_{D^b{\rm MHM}(X)}(\Q^H_X[n], \gr^W_{n+i}\H^{i}(\Q^H_X[n])[-i])=0
$$
when $i<0$. This implies that
$$
\gr^W_{n+i}\H^{i}(\K^\bullet_X)=0
$$
for all $i$, and thus, $\K^\bullet_X$ is of weight $\le n-1$.
\end{proof}

\begin{rmk}[{\bf Weber's theorem}]\label{rmk:Weber}
When $X$ is projective, the long exact sequence of hypercohomology associated to \eqref{eqn: Q to IC triangle} induces
$$
\to\HH^i(X,\K^\bullet_X[-n])\to H^i(X,\Q)\to \mathit{IH}^i(X,\Q)\to.
$$
Using the fact that $\K^\bullet_X$ is of weight $\le n-1$, we recover \cite[Theorem 1.8]{Weber04}, which states that
$$
\mathrm{kernel}\big( H^i(X,\Q)\to \mathit{IH}^i(X,\Q) \big)=W_{i-1}H^i(X,\Q).
$$
See also \cite[Exercise 11.3.9]{Maxim19}.
\end{rmk}

\smallskip

We next describe the behavior of $\K^\bullet_X$ upon restriction to a very general hyperplane section, analogous to the behavior of the intersection complex described in 
\cite[Proposition 4.17]{KS21} (cf. also \cite[Section 6.2]{Park23}).

\begin{lem}
\label{lem: cutting hyperplane of K}
Assume that $X$ is embedded in a smooth quasi-projective variety $Y$. For a very general hyperplane $\iota:L\hookrightarrow Y$, we have an isomorphism
$$
\iota^*\K^\bullet_X\isom \K^\bullet_{X\cap L}[1],
$$
and there exists a distinguished triangle
$$
N^*_{L/Y}\otimes_{\O_L}\gr^F_{p+1}\DR\left(\K^\bullet_{X\cap L}\right)\to \O_L\tensor_{\O_Y}\gr^F_{p}\DR\left(\K^\bullet_X\right)\to \gr^F_{p}\DR\left(\K^\bullet_{X\cap L}\right)[1]\xrightarrow{+1}
$$
in $D^b_{coh}(L, \O_L)$, where $N^*_{L/Y}$ is the conormal bundle of $L\subset Y$.
\end{lem}

\begin{proof}
Since $L$ is chosen to be very general, the embedding $\iota: L\hookrightarrow Y$ is non-characteristic for $\IC^H_X$. This implies that
$$
\iota^*\IC_X^H\isom \IC^H_{X\cap L}[1]
$$
by Saito's theory \cite{Saito90} (or see \cite[Theorem 8.3]{Schnell16}). Therefore, the pullback of the distinguished triangle \eqref{eqn: Q to IC triangle} yields another distinguished triangle
$$
\iota^*\K_X^\bullet\to\Q_{X\cap L}^H[n]\to\IC_{X\cap L}^H[1]\xrightarrow{+1}.
$$
The morphism $\Q_{X\cap L}^H[n]\to\IC_{X\cap L}^H[1]$ is an isomorphism on the smooth locus of $X\cap L$, which implies that this morphism is the natural morphism 
$\Q_{X\cap L}^H[n-1]\to\IC_{X\cap L}^H$ shifted by $[1]$. Therefore, we have an isomorphism
$$
\iota^*\K^\bullet_X\isom \K^\bullet_{X\cap L}[1].
$$

Suppose $\K_X^\bullet$ is represented by a complex of mixed Hodge modules $\M^\bullet$ on $Y$. Since $\iota:L\hookrightarrow Y$ is chosen to be very general, Lemma \ref{lem:hyperplane-de Rham} applies to each term of the complex $\M^\bullet$. This induces the short exact sequence
$$
0\to N^*_{L/Y}\otimes_{\O_L}\gr^F_{p+1}\DR\left(\M_L^\bullet\right)\to \O_L\tensor_{\O_Y}\gr^F_{p}\DR\left(\M^\bullet\right)\to \gr^F_{p}\DR\left(\M_L^\bullet\right)[1]\to 0.
$$

From the isomorphism $\iota^*\K^\bullet_X\isom \K^\bullet_{X\cap L}[1]$, the object $\K^\bullet_{X\cap L}$ is represented by the complex $\M_L^\bullet$. Therefore, we obtain the distinguished triangle in the statement.
\end{proof}

The next natural step is to apply  the graded de Rham functor $\gr^F_{-k}\DR(\, \cdot \,)$ to \eqref{eqn: Q to IC triangle}. This leads to the following distinguished triangle
\begin{equation}
\label{eqn: grdr Q to IC}
\gr^F_{-k}\DR(\K_X^\bullet)\to \DB^k_X[n-k]\to I\DB^k_X[n-k]\xrightarrow{+1}
\end{equation}
in ${\bf D}^b_{\rm coh}(X)$. In other words, 
$$
\gr^F_{-k}\DR(\K_X^\bullet) [k - n - 1] = {\rm Cone} \big(\gamma_k \colon \DB^k_X \to I\DB^k_X\big).
$$

We obviously have the equivalence 
$$D_m \iff \gr^F_{-p}\DR(\K_X^\bullet) = 0 \,\,\,\,\,\,{\rm for ~all} \,\,\,\, 0 \le p \le k,$$
in which case, according to Lemma \ref{lem:*k-coincidence}, when $X$ is projective we have  $H^{p,q} (X) =\mathit{IH}^{p,q} (X) $ for all $0\le q\le n$.

Understanding the vanishing of various cohomologies of $\gr^F_{-k}\DR(\K_X^\bullet)$, or conditions on the dimension of the support of these cohomologies, provides a useful tool towards refining this type of implication in order to include further Hodge numbers. We first analyze vanishing statements.

\begin{lem}
\label{lem: vanishing of grdr K}
We have 
$$\gr^F_{-k}\DR(\K_X^\bullet)=0 \,\,\,\,\,\,{\rm for~ all} \,\,\,\, k\notin[0,n-1].$$ 
When $k\in[0,n-1]$, we have
$$
\H^{i}\gr^F_{-k}\DR\left(\K_X^\bullet\right)=0 \,\,\,\,\,\,{\rm for~ all} \,\,\,\, i\le k-n \,\,\,\,{\rm  and}  \,\,\,\, i>0.
$$
In other words, the cohomologies of $\gr^F_{-k}\DR\left(\K_X^\bullet\right)$ are supported on degrees $[k-n+1,0]$.
\end{lem}

\begin{proof}
It is clear from the definition that $\DB^k_X=0$ for $k\notin[0,n]$. Likewise, we have $I\DB^k_X=0$ for $k\notin[0,n]$, for instance from Lemma \ref{prop: 0th cohomology of IDB}.
Therefore, $\gr^F_{-k}\DR(\K_X^\bullet)=0$ for all $k\notin[0,n]$. The vanishing for $k=n$ in the first statement of the Lemma is contained in the rest of the proof.

\noindent
{\em Claim 1.} For $k\in [0,n]$, we have $\H^{i}\gr^F_{-k}\DR\left(\K_X^\bullet\right)=0$ for all $i>0$.

Recall from the beginning of Section \ref{scn:IC} that the perverse cohomologies of $\Q^H_X[n]$ are supported on non-positive degrees, and the morphism $\gamma_X$ is the composition of quotient morphisms
$$
\Q^H_X[n]\to \H^0\Q^H_X[n]\to \IC^H_X=\gr^W_n  \H^0  \Q^H_X[n].
$$
Hence, the cohomologies of $\K^\bullet_X$ are supported in non-positive degrees. Using the canonical truncation at degree $0$, the object $\K^\bullet_X$ can be represented by a complex of mixed Hodge modules supported in non-positive degrees, hence the cohomologies of $\gr^F_{-k}\DR\left(\K_X^\bullet\right)$ are also supported in non-positive degrees.

 \noindent
{\em Claim 2.}  For $k\in [0,n]$, we have $\H^{i}\gr^F_{-k}\DR\left(\K_X^\bullet\right)=0$ for all $i\le k-n$.

From Lemma \ref{prop: 0th cohomology of IDB}, we have $\H^iI\DB^k_X=0$ for $i<0$, and we know that $\H^i\DB_X^k=0$ for $i<0$ as well. Using the distinguished triangle \eqref{eqn: grdr Q to IC}, it suffices to prove that the morphism
$$
\H^0\DB^k_X\to \H^0I\DB_X^k
$$
is an injection. But this is clear from the fact that $\H^0\DB^k_X$ is torsion-free (see Section \ref{scn:DB}), as the two sheaves are naturally isomorphic on the smooth locus of $X$.

Claims $1$ and $2$ complete the proof of the Lemma. 
\end{proof}

In order to say more, we make use of the following simple general fact:

\begin{lem}
\label{lem: dim supp cohomology vanishing}
Let $X$ be a projective variety and $\mathcal C^\bullet\in {\bf D}^b_{\rm coh}(X)$ be an object in the derived category of coherent sheaves on $X$. If 
for some non-negative integer $m$ we have
$$
\dim \left(\Supp\left(\H^{i}\mathcal C^\bullet\right)\right)\le m-i \quad {\rm for~all}\,\,\,\, i,
$$
then
$$
\HH^q\left(X,\mathcal C^\bullet\right)=0 \,\,\,\,\,\, {\rm for~ all} \,\,\,\,q>m.
$$
\end{lem}

\begin{proof}
Consider the spectral sequence
$$
E_2^{i,j}=\HH^i(X,\H^{j}\mathcal C^\bullet) \Longrightarrow \HH^{i+j}(X,\mathcal C^\bullet).
$$
If $i+j>m$, then $E_2^{i,j}=0$ by our hypothesis. Therefore, $\HH^{q}(X,\mathcal C^\bullet)=0$ for $q>m$.
\end{proof}

In particular, if for fixed $p$ and $k$ we happen to have
\begin{equation}
\label{eqn: bound on dim supp cohomology}
\dim \left(\Supp\left(\H^{-n+p+i}\gr^F_{-p}\DR\left(\K_X^\bullet\right)\right)\right)\le n-i-k-1 \quad {\rm for~all}\,\,\,\, i,
\end{equation}
then 
$$\HH^q \big(X,\gr^F_{-p}\DR\left(\K_X^\bullet\right)[p - n]\big)=0 \,\,\,\,\,\,{\rm for~ all} \,\,\,\, q\ge n-k.$$ 
Therefore, by the long exact sequence of hypercohomologies induced by \eqref{eqn: grdr Q to IC}, we have
$$
\HH^q (X,\DB^p_X)=\HH^q(X,I\DB_X^p) \,\,\,\,\,\, {\rm for~all} \,\,\,\, q\ge n-k.
$$

The crucial point is that \eqref{eqn: bound on dim supp cohomology} for a certain range of $p$ is in fact sufficient to establish this for all $p$, as described in the following:

\begin{prop}
\label{prop: dim supp inequality of K}
Suppose that for a non-negative integer $k$, \eqref{eqn: bound on dim supp cohomology} is satisfied for all $0\le p\le k$. Then the same inequality 
\eqref{eqn: bound on dim supp cohomology} holds  in fact for all $0\le p\le n$, that is
$$
\dim \left(\Supp\left(\H^{-n+p+i}\gr^F_{-p}\DR\left(\K_X^\bullet\right)\right)\right)\le n-i-k-1, \quad {\rm for~all~} i \mathrm{\;~and~ all\;}p\in [0,n].
$$
If furthermore $X$ is projective, then we consequently have
$$
\HH^q (X,\DB^p_X)=\HH^q(X,I\DB_X^p)
$$
for all $0\le p\le n$ and $n-k\le q\le n$, so in particular $\underline h^{p,q} (X) =I\underline h^{p,q} (X)$ in this range.
\end{prop}

We follow the convention that the dimension of the empty set is $-\infty$.

\begin{proof}
The claims in the projective case are immediate from the argument that precedes the statement of the Proposition. Therefore it suffices 
to prove the following claim, which we do by induction on the dimension $n$.

\noindent
{\em Claim.} If \eqref{eqn: bound on dim supp cohomology} holds for all $0\le p\le k$, then \eqref{eqn: bound on dim supp cohomology} holds for all $0\le p\le n$.

The base case $n=k$ is trivial, so we assume $n > k$. Note that the statement \eqref{eqn: bound on dim supp cohomology} is local, so  we may assume that $X$ is embeddable in a smooth quasi-projective variety $Y$. We take a very general hyperplane section $\iota:L\hookrightarrow Y$.
From the distinguished triangle
$$
N^*_{L/Y}\otimes_{\O_L}\gr^F_{-p+1}\DR\left(\K^\bullet_{X\cap L}\right)\to \O_L\tensor_{\O_Y}\gr^F_{-p}\DR\left(\K^\bullet_X\right)\to \gr^F_{-p}\DR\left(\K^\bullet_{X\cap L}\right)[1]\xrightarrow{+1}
$$
provided by Lemma \ref{lem: cutting hyperplane of K}, we have a long exact sequence of cohomologies
\begin{multline}
\label{eqn: LES of hyperplane of K}
\to \O_L\tensor_{\O_Y}\H^{-n+p+i}\gr^F_{-p}\DR\left(\K_X^\bullet\right)\to\H^{-(n-1)+p+i}\gr^F_{-p}\DR\left(\K^\bullet_{X\cap L}\right)\\
\to N^*_{L/Y}\otimes_{\O_L}\H^{-(n-1)+p+i}\gr^F_{-p+1}\DR\left(\K^\bullet_{X\cap L}\right)\to
\end{multline}
Note that we have $\gr^F_{1}\DR\left(\K^\bullet_{X\cap L}\right)=0$ by Lemma \ref{lem: vanishing of grdr K}. Combined with the sequence above, this implies that \eqref{eqn: bound on dim supp cohomology} holds for $X\cap L$ and $p=0$, that is
$$
\dim \left(\Supp\left(\H^{-(n-1)+i}\gr^F_{0}\DR\left(\K_{X\cap L}^\bullet\right)\right)\right)\le (n-1)-i-k-1, \quad {\rm for~all}\,\,\,\, i.
$$
Since \eqref{eqn: bound on dim supp cohomology} holds for $X$ and $0 \leq p \leq k$, it is then straightforward to see, by induction on $p$ and using \eqref{eqn: LES of hyperplane of K}, that \eqref{eqn: bound on dim supp cohomology} also holds for $X \cap L$ and $0 \leq p \leq k$.

Thus, by the induction hypothesis on $n$, we conclude that \eqref{eqn: bound on dim supp cohomology} holds for $X \cap L$ and $0 \leq p \leq n-1$. Using this and \eqref{eqn: LES of hyperplane of K}, we have
$$
\dim \left(\Supp\left(\H^{-n+p+i}\gr^F_{-p}\DR\left(\K_{X}^\bullet\right)\right)\right)\le \max\left\{n-i-k-1,0\right\}, \quad {\rm for~all}\,\,\,\, i,
$$
for all $0\le p\le n$. Therefore, it suffices to prove
$$
\H^{-n+p+i}\gr^F_{-p}\DR\left(\K_{X}^\bullet\right)=0
$$
when $n-i-k-1<0$ and $p>k$. In this case, we have $-n+p+i>0$, and thus Lemma \ref{lem: vanishing of grdr K} completes the proof.
\end{proof}

\subsection{Hodge symmetry via condition $D_m$}\label{sec: Hodge symmetry for singular varieties}
This section is devoted to the proof of the following theorem, which thanks to Proposition \ref{prop:krat*} implies the first statement of Theorem \ref{thm:Hodge-symmetry-krat} in the Introduction, answering in particular Laza's question on the symmetry of Hodge-Du Bois numbers for higher rational singularities.

\begin{thm}\label{thm: Hodge symmetry for k-rational}
Let $X$ be a projective equidimensional variety of dimension $n$, satisfying condition $D_m$.  Then we have natural identifications
$$
H^{p,q} (X) = \mathit{IH}^{p,q}(X), \,\,\,\,H^{q,p} (X) = \mathit{IH}^{q,p}(X),$$
$$H^{n-p,n-q} (X) = \mathit{IH}^{n-p,n-q}(X), \,\,\,\,H^{n-q,n-p} (X) = \mathit{IH}^{n-q,n-p}(X)
$$
for all $0\le p\le m$ and $0\le q\le n$. In particular, we have 
$$
\underline h^{p,q} (X) =\underline h^{q,p} (X) =\underline h^{n-p,n-q} (X) =\underline h^{n-q,n-p} (X)
$$
for all such $p$ and $q$.
\end{thm}

\begin{rmk}
The proof will also give the extra conclusion 
$$H^{p,q} (X) = \mathit{IH}^{p,q}(X), \,\,\,\,\,\,{\rm for}\,\,\,\,p = n - m - 1$$
and all $q$, under the same hypothesis.
\end{rmk}

When $X$ has $m$-rational singularities, the equality of the first three numbers in the last statement was shown in \cite[Theorem 3.24]{FL22} when $X$ is a local complete intersection, and in \cite[Corollary 4.1]{SVV23} in general.

\smallskip

\begin{rmk}[{Rational singularities}]\label{rmk:symm-rational}
Under the assumption $\DB_X^0 \simeq I\DB_X^0$, we have 
$$
\underline h^{0,q}(X)=\underline h^{q,0}(X)=\underline h^{n,n-q}(X)=\underline h^{n-q,n}(X)
$$
for all $0\le q\le n$,  i.e. the symmetry of the boundary of the Hodge-Du Bois diamond.  Note that $\H^0\DB_X^0\simeq \H^0I\DB_X^0$ implies by Lemma \ref{prop: 0th cohomology of IDB} that  the seminormalization $X^{\rm sn}$ of $X$ is equal to the normalization of $X$, as $\H^0\DB_X^0 \simeq \O_{X^{\rm sn}}$  by \cite[Proposition 5.2]{Saito00}.

Note that equality of all four numbers is new even under the stronger assumption of $X$ having rational singularities, which is equivalent to 
$\O_X \simeq \DB_X^0 \simeq I\DB_X^0$ (see e.g. \cite[Proposition 4.4]{Park23}). 
\end{rmk}

We start with some preliminaries, where $X$ is not necessarily assumed to be projective, showing that there are more fundamental symmetries lying behind 
those of Hodge numbers.

Note first that for an object $\M^\bullet\in D^b{\rm MHM} (X)$ in the derived category of mixed Hodge modules on $X$, it is easy to check that the condition
$$
\gr^F_{-p}\DR(\M^\bullet)=0 \quad \mathrm{for~all}\quad p\le k
$$
is equivalent to the condition
$$
\gr^F_{-p}\DR(\H^i\M^\bullet)=0  \quad \mathrm{for~all}\quad p\le k \quad {\rm and ~all} \quad i.
$$
This becomes apparent once we consider dual statements: $\gr^F_{p}\DR(\dual\M^\bullet)=0$ for $p\le k$ if and only if the index of the first nonzero term in the Hodge filtration of $\H^i(\dual\M^\bullet)$ is at least $k+1$ for all $i$. When applied to $\K^\bullet_X$, this observation leads to the following:

\begin{prop}
\label{prop: 0 to k then n-k-1 to n}
Let $X$ be an equidimensional variety of dimension $n$, satisfying condition $D_m$. 
Then we additionally have:

\noindent
(i)  $\DB^p_X=I\DB^p_X$ for $n-m-1\le p\le n$.

\noindent
(ii) $\H^i(\K^\bullet_X)=0$ for $i\le -n+2(m+1)$.
\end{prop}

\begin{proof}
The assumption means that $\gr^F_{-p}\DR(\K^\bullet_X)=0$ for $p\le m$. As explained in the previous paragraph, this implies that
$$
\gr^F_{-p}\DR(\H^i\K^\bullet_X)=0 \quad \mathrm{for~all}\quad p\le m \quad {\rm and ~all} \quad i.
$$
By Proposition \ref{prop: weight of K}, the $i$-th cohomology $\H^i\K^\bullet_X$ has weight $\le i+n-1$. Therefore we also have
$$
\gr^F_{-p}\DR(\H^i\K^\bullet_X)=0 \quad \mathrm{for~all}\quad p\ge i + n -1- m \quad {\rm and ~all} \quad i
$$
by Proposition \ref{prop: range of nonzero grdr}. Recall in addition that $\H^i\K^\bullet_X=0$ for $i>0$, by definition. We deduce that 
$\gr^F_{-p}\DR(\K^\bullet_X)=0$ for $p\ge n-m-1$, since the same holds for all the cohomologies of $\K^\bullet_X$. This is equivalent to
$$
\DB^p_X=I\DB^p_X\quad \mathrm{for} \quad n-m-1\le p\le n.
$$

Additionally, the argument shows that $\gr^F_{-p}\DR(\H^i\K^\bullet_X)=0$ for all $p$, if $i+n-m-2\le m$. As a consequence, we have $\H^i(\K^\bullet_X)=0$ for $i\le -n+2(m+1)$.
\end{proof}

As an immediate consequence, we record the crucial fact that condition $D_m$ leads to behavior that is roughly ``twice as good" when it comes to measuring the 
rational homology manifold condition, or the local cohomological defect. This implies in particular the second part of Theorem  \ref{thm:Hodge-symmetry-krat} in the Introduction.

\begin{cor}
\label{cor: lcdef of k-rational}
Let $X$ be an equidimensional variety of dimension $n$, satisfying condition $D_m$.  Then:

\noindent
(i) $X$ is a rational homology manifold away from a closed subset of codimension at least $2m+3$.

\noindent
(ii)  $\lcdef(X)\le \max\left\{n-2m-3,0\right\}$.
\end{cor}

\begin{proof}
By Proposition \ref{prop: 0 to k then n-k-1 to n}, when $n\le 2m+2$ we have $\K^\bullet_X=0$, that is $X$ is a rational homology manifold. Therefore we also 
have $\lcdef (X) = 0$.

Suppose now that $n\ge 2m+3$. This statement is local, so we assume that $X$ is embeddable in a smooth quasi-projective variety. Applying Lemma \ref{lem: cutting hyperplane of K}, condition $D_m$ is preserved under taking very general hyperplane sections. After taking $n-2m+2$  such hyperplane sections, the resulting $(2m+2)$-dimensional variety is therefore a rational homology manifold. Using again Lemma \ref{lem: cutting hyperplane of K} for the behavior of $\K^\bullet_X$ under taking very general hyperplane sections, this implies that $X$ is a rational homology manifold away from a closed subset of codimension at least $2m+3$.

Recall next from Proposition \ref{prop: 0 to k then n-k-1 to n} that $\H^i(\K^\bullet_X)=0$ for $i\le -n+2(m+1)$. Using the distinguished triangle ($\ref{eqn: Q to IC triangle}$) 
and the fact that $\IC_X^H$ only has cohomology in degree $0$, this implies that 
$$\H^i(\Q^H_X[n])=0 \,\,\,\,\,\,{\rm  for~all} \,\,\,\, i\le \min\left\{-n+2(m+1),-1\right\}.$$
According to Theorem \ref{thm:lcdef-top}, this gives $\lcdef(X)\le \max\left\{n-2m-3,0\right\}$.
\end{proof}

We are now ready to prove the main result of the section.

\begin{proof}[Proof of Theorem \ref{thm: Hodge symmetry for k-rational}]
We consider the long exact sequence of hypercohomology associated to the distinguished triangle ($\ref{eqn: Q to IC triangle}$), that is
$$
\cdots \to \HH^k(X,\K^\bullet_X[-n])\to H^k(X,\Q)\to \mathit{IH}^k (X,\Q)\to \HH^{k+1}(X,\K^\bullet_X[-n])\to \cdots 
$$

By Proposition \ref{prop: weight of K}, the object $\K^\bullet_X[-n]$ is of weight $\le -1$, hence the mixed Hodge structure $\HH^k(X,\K^\bullet_X[-n])$ is of weight $\le k-1$ (see \cite[4.5.2]{Saito90}). Since
$$
\gr^F_{-p}\HH^k(X,\K^\bullet_X[-n])=0 \,\,\,\,\,\,{\rm for} \,\,\,\,p \le m,
$$
the same then holds for $p\ge k-1-m$ by Proposition \ref{prop: range of nonzero grdr} and the discussion preceding it. This implies that
$$
\gr^F_{-p}H^k(X,\C) = \gr^F_{-p}\mathit{IH}^k(X,\C)
$$
for $p\le m$ or $p\ge k-m$. Hence, $H^{p,q} (X)  = \mathit{IH}^{p,q} (X)$ for $p\le m$ or $q\le m$.

By Proposition \ref{prop: 0 to k then n-k-1 to n}, we have $H^{p,q} (X) = \mathit{IH}^{p,q} (X) $ for $p\ge n-m-1$. By Proposition \ref{prop: dim supp inequality of K}, we have 
$H^{p,q} (X) = \mathit{IH}^{p,q} (X) $ for $q\ge n-m$. This shows all the natural isomorphisms at the level of spaces.

Finally, combining all this with the Hodge symmetry of intersection cohomology
$$
I\underline h^{p,q}(X)=I\underline h^{q,p}(X)=I\underline h^{n-p,n-q}(X)=I\underline h^{n-q,n-p}(X),
$$
we obtain the desired conclusion.
\end{proof}

\section{Hodge-Du Bois symmetry II: Weak Lefschetz theorems}

\subsection{Weak Lefschetz using the local cohomological defect}\label{scn:weak-Lef}
The classical Lefschetz hyperplane theorem, also known as the weak Lefschetz theorem, states that when $X$ is a smooth projective variety of dimension $n$ and $D$ is an ample effective Cartier divisor, the restriction map $H^i(X,\Q)\to H^i(D,\Q)$ is an isomorphism for $i\le n-2$ and injective when $i=n-1$.

We first record a generalization of this statement that uses the local cohomological defect. As mentioned in the Introduction, this result should be seen as the joint effort of a number of authors, starting with Ogus. The proof follows from the Riemann-Hilbert correspondence for local cohomology and the Artin vanishing theorem \cite[Theorem 4.1.1]{BBD}.

\begin{proof}[Proof of Theorem \ref{thm:Lefschetz-lcdef}]
Denote $U:=X\sm D$ and $j:U\hookrightarrow X$. According to Theorem \ref{thm:lcdef-top}, we have $\Q_U\in {}^p\!D^{\ge n-\lcdef(U)}_c(X,\Q)$.
%%depending on how you define lcdef, this might be immediate.. probably state this right after the definition
We consider the distinguished triangle
$$
j_!\Q_U\to \Q_X\to \Q_D\xrightarrow{+1}
$$
and its long exact sequence of hypercohomology
$$
\cdots \to H_c^i(U,\Q)\to H^i(X,\Q)\to H^i(D,\Q)\to \cdots
$$
If we denote by $a_U\colon U\to \mathrm{pt}$ the constant map to a point, the cohomology with compact support $H^i_c(U,\Q)$ is the $i$-th cohomology of the object ${a_U}_!\Q_U$.
Since $U$ is affine, ${a_U}_!$ is a left perverse exact functor by the generalized version of the Artin vanishing theorem \cite[Theorem 4.1.1]{BBD}. This implies that 
$$H_c^i(U,\Q)=0 \,\,\,\,\,\,{\rm  for} \,\,\,\, i\le n-1-\lcdef(U),$$
which thanks to the exact sequence above is equivalent to the assertion of the theorem.
\end{proof}

\begin{rmk}[{\bf Another approach via vanishing for Du Bois complexes}]\label{rmk:DB-vanishing}
It is well known that, at least for well-behaved divisors $D$ in a smooth variety $X$, the weak Lefschetz theorem is essentially equivalent to the Kodaira-Nakano vanishing theorem, via the Hodge decomposition. The same happens in the present context, at least if we use $\lcdef (X)$.

If $D$ is a general hyperplane section of $X$, then for each $p$ there is an exact triangle
$$\DB_D^{p -1} (-D) \to \DB_X^p |_{D} \to \DB_D^p \xrightarrow{+1}.$$
See \cite[Lemma 3.2]{SVV23}. Using this and the octahedral axiom, it is not hard to see that we have an exact triangle 
$$C_{X, D}^p \to \DB_X^p \to \DB_D^p \xrightarrow{+1},$$
where in turn the object $C_{X, D}^p$ sits in an exact triangle
\begin{equation}\label{eqn:DB-induction-triangle}
\DB_X^{p} (-D) \to  C_{X, D}^p \to  \DB_D^{p -1} (-D)  \xrightarrow{+1}.
\end{equation}
On the other hand, there is a dual version of the Nakano vanishing theorem for Du Bois complexes:\footnote{The subtlety here is that, unlike in the smooth case, this vanishing theorem is not the Serre dual of the typical Nakano-type vanishing theorem for Du Bois complexes, saying that $H^q (X, \DB_X^p \otimes L) = 0$ for 
$p + q > n$.}

\begin{thm}[{\cite[Theorem 5.1]{PS24}}]\label{thm:DB-vanishing}
If $L$ is an ample line bundle on a projective variety $X$, then 
$$\HH^q (X, \DB_X^p \otimes L^{-1}) = 0 \,\,\,\,\,\,{\rm for ~all }\,\,\,\, p + q \le  n -1 - \lcdef (X).$$
\end{thm}

Since by generality we have $\lcdef (D) \le \lcdef (X)$, applying this vanishing theorem on both extremes of ($\ref{eqn:DB-induction-triangle}$) gives 
$$\HH^q (X, C_{X, D}^p) = 0 \,\,\,\,\,\,{\rm for ~all }\,\,\,\, p + q \le  n -1 - \lcdef (X),$$
which by the previous triangle is equivalent to the conclusion of Theorem  \ref{thm:Lefschetz-lcdef}.
\end{rmk}

\begin{rmk}[{\bf Quasi-projective case}]
Using \cite[Theorem 5.1.2]{dCM10} or \cite[Theorem 2.0.3]{deCataldo12} in place of the Artin vanishing theorem in the proof of Theorem \ref{thm:Lefschetz-lcdef}, we obtain a version for quasi-projective varieties.

\begin{thm}
\label{thm: Lefschetz hyperplane theorem for quasi-projective}
Let $X$ be an equidimensional quasi-projective variety of dimension $n$, and let $D \subset X$ be a general hyperplane section. Then the restriction
$$
H^i(X,\Q)\to H^i(D,\Q)
$$
is an isomorphism for $i\le n-2-\lcdef(X\smallsetminus D)$ and injective when $i=n-1-\lcdef(X\smallsetminus D)$. 
\end{thm}
\end{rmk}

\begin{rmk}[{\bf Lefschetz hyperplane theorem for intersection cohomology}]
We also recall, for repeated use, the fact that intersection cohomology is known to satisfy the strong version of the Lefschetz hyperplane theorem.

\begin{thm}[{\cite[Section II.6.10]{GM88}}]
\label{thm: Lefschetz hyperplane theorem for intersection cohomologies}
Let $X$ be an equidimensional (quasi-)projective variety of dimension $n$, and let $D \subset X$ be a general hyperplane section. Then the natural map of (mixed) $\Q$-Hodge structures given by restriction
$$
\mathit{IH}^i(X,\Q)\to \mathit{IH}^i(D,\Q)
$$
is an isomorphism for $i\le n-2$ and injective when $i=n-1$. 
\end{thm}

Stated like this, the theorem incorporates the Hodge structure on intersection cohomology, coming from the (pure) Hodge module structure on $\IC_X^H$ in Saito's work, mentioned previously. In this setting, there is an isomorphism 
\begin{equation}\label{eqn:restriction-IC}
\iota^*\IC^H_X \simeq \IC^H_D[1]
\end{equation}
when $D$ is transverse to a Whitney stratification of $X$, and Theorem \ref{thm: Lefschetz hyperplane theorem for intersection cohomologies} can be seen as a quick 
consequence of this fact, precisely along the lines of the proof of Theorem \ref{thm:Lefschetz-lcdef} and Theorem \ref{thm: Lefschetz hyperplane theorem for quasi-projective}.
\end{rmk}

\smallskip 
Theorem \ref{thm:Lefschetz-lcdef} has some interesting consequences regarding the purity of Hodge structures.

\begin{cor}\label{cor:purity}
Let $X$ be an equidimensional projective variety with $\codim~ \Sing(X) \ge k+1$. Then the Hodge structure on $H^i (X, \Q)$ is pure for $i \le k - \lcdef (X)$. Moreover, we have 
$$H^i (X, \Q) \simeq \mathit{IH}^i (X, \Q) \,\,\,\,{\rm for} \,\,\,\, i \le k -1 - \lcdef (X),$$
so in particular $H^i (X, \Q)$ is Poincar\'e dual to $H^{2n- i} (X, \Q)$ for such $i$.
\end{cor}
\begin{proof}
Note that if $U$ is any open set in $X$, we have $\lcdef (U) \le \lcdef (X)$, and if $D$ is a general hyperplane section of $X$ we have 
$\lcdef (D) \le \lcdef (X)$. Thus by cutting with $(n- k)$ general hyperplane sections and applying the theorem repeatedly, we can reduce to the case when $D$ is smooth, 
when $H^i (D, \Q)$ is a pure Hodge structure. Moreover, in this case we have $H^i (D, \Q)\simeq \mathit{IH}^i (D, \Q)$, but we always have $\mathit{IH}^i (X, \Q)\simeq \mathit{IH}^i (D, \Q)$
by the weak Lefschetz theorem for intersection cohomology.

Note that since the singular locus of $X$ has dimension at most $n - k -1$, we also have the isomorphism $H^i (X, \Q) \simeq \mathit{IH}^i (X, \Q)$ for $i \ge 2n -k$, which
implies the last assertion as intersection cohomology satisfies Poincar\'e duality.
\end{proof}

\begin{ex}[{\bf Isolated singularities}]
When $X$ has isolated singularities we can take $k = n -1$; if moreover $\lcdef (X) = 0$, we obtain that all $H^i (X, \Q)$ carry pure Hodge structure, 
except perhaps $H^n (X, \Q)$.

According to Example \ref{ex:def=0}, this is the case for instance when $X$ is any Cohen-Macaulay threefold, or a Cohen-Macaulay fourfold which is locally 
analytically $\Q$-factorial, with isolated singularities.  

If ${\rm depth} (\O_X) \ge 3$ in any dimension, then $\lcdef (X) \le n-3$ (again by Example \ref{ex:def=0}), and therefore $H^2 (X, \Q)$ carries pure Hodge structure. (Cf. the Appendix by Srinivas to 
\cite{GW2018} for 
this statement for isolated Cohen-Macaulay singularities.)
\end{ex}

\begin{ex}[{\bf Normal and rational singularities}]\label{ex:normal-H1}
When $X$ is a projective normal variety, so that ${\rm depth} (\O_X) \ge 2$ and hence $\lcdef (X) \le n-2$, the iterative application of Theorem \ref{thm:Lefschetz-lcdef} implies the injection
$$
H^1(X,\Q)\to H^1(C,\Q)
$$
where $C$ is a general complete intersection curve in $X$. Since $C$ is smooth by the normality of $X$, this recovers the well-known fact that $H^1(X,\Q)$ is a pure Hodge structure of weight $1$ (see e.g. \cite{Saito18}). As in the previous example, when $X$ has rational singularities and $\dim X \ge 3$, 
we have $\lcdef (X) \le n- 3$, hence a similar procedure recovers the folklore fact that $H^2 (X,\Q)$ is a pure Hodge structure of weight $2$, using the well-known statement that surfaces with rational singularities are rational homology manifolds.

In the same vein, when $X$ has rational singularities we also obtain that $H^1 (X, \Q)$ and $H^{2n-1} (X, \Q)$ are Poincar\'e dual. For klt threefolds, this was shown in 
\cite[Theorem 5.13(1)]{GW2018}.
\end{ex}

\subsection{Weak Lefschetz using condition $D_m$}\label{scn:WL2}
Taking a different approach towards Lefschetz-type results, recall from Corollary \ref{cor: lcdef of k-rational} that if $U = X \smallsetminus D$ satisfies condition $D_m$, which we recall means that the canonical map 
$\DB^p_{U} \to I\DB^p_{U}$ is an isomorphism for $0\le p\le m$, then we have the upper bound 
\begin{equation}\label{eqn:lcdef-repeat}
\lcdef (U) \le   \max\left\{n-2m-3,0\right\}.
\end{equation}
Theorem \ref{thm:Lefschetz-lcdef} then says that the Lefschetz hyperplane morphism is an isomorphism for $0\le i\le \min\left\{n-2, 2m+1\right\}$ and injective for $i=\min\left\{n-1, 2m+2\right\}$.

The following more precise version of Theorem \ref{thm:Lefschetz-hyp-krat} in the Introduction gives however better bounds under this assumption, when $n \ge 2m + 3$,
using the techniques of 
 Sections \ref{scn:RHM} and \ref{sec: Hodge symmetry for singular varieties}. Note first that the restriction map $H^i(X,\Q)\to H^i(D,\Q)$ is a morphism of mixed Hodge structures, hence it also induces a restriction map 
$$H^{p,q}(X)=\gr^F_{-p}H^{p+q}(X,\C) \to H^{p,q}(D)=\gr^F_{-p}H^{p+q}(D,\C).$$

\smallskip

\begin{thm}\label{thm:Lefschetz-krat}
Let $X$ be an equidimensional projective variety of dimension $n$,  and let $D$ be an ample effective Cartier divisor on $X$. Suppose that $U = X \smallsetminus D$ satisfies condition $D_m$.  Then the restriction map
$$
H^{p,q}(X)\to H^{p,q}(D)
$$
is an isomorphism when $p+q\le n-2$ and $\min\left\{p,q\right\}\le m$. Moreover, it is  injective when $p+q\le n-1$ and  $\min\left\{p,q-1\right\}\le m$.
\end{thm}

For example, while Theorem  \ref{thm:Lefschetz-lcdef} and Theorem \ref{thm:Lefschetz-krat} are equivalent for threefolds with $m=0$, if $X$ is a fourfold and $m = 0$ (so for instance a fourfold with rational singularities), then in addition the spaces $H^{0,3} (X)$, $H^{2,1} (X)$ and $H^{3, 0} (X)$ satisfy weak Lefschetz. More and more spaces like this appear as the dimension goes up.

\begin{proof}
As in the proof of Theorem  \ref{thm:Lefschetz-lcdef}, we consider the long exact sequence of  mixed $\Q$-Hodge structures
$$
\cdots \to H_c^i(U,\Q)\to H^i(X,\Q)\to H^i(D,\Q)\to \cdots
$$
Denote $H^{p,q}_c(U):=\gr^F_{-p}H^{p+q}_c(U,\C)$. We then have the corresponding long exact sequence
$$
\cdots \to H_c^{p,q}(U)\to H^{p,q}(X)\to H^{p,q}(D)\to H_c^{p,q+1}(U)\to \cdots 
$$
Therefore, it suffices to prove the following claim.

\noindent
{\em Claim.} When $p+q\le n-1$, we have $H^{p,q}_c(U)=0$ for $p\le m$ or $q\le m +1$.

Recall that by Saito \cite[Corollary 4.3]{Saito00}, the classical mixed Hodge structure of $H_c^i(U,\Q)$ is the same as the mixed Hodge structure obtained from the theory of mixed Hodge modules. We consider the distinguished triangle \eqref{eqn: Q to IC triangle}  for $U$, shifted by $[-n]$:
$$
\K^\bullet_U[-n]\to\Q^H_U\to\IC^H_U[-n]\xrightarrow{+1}.
$$
Denote by $a_U:U\to \mathrm{pt}$ the constant map to a point. Since $a_U$ is an affine morphism and $\IC_U[-n]\in {}^p\!D^{\ge n}_c(X,\Q)$, we have 
$$H^i({a_U}_!\IC^H_U[-n])=0 \,\,\,\, {\rm for} \,\,\,\,i\le n-1$$ 
by the Artin vanishing theorem \cite[Theorem 4.1.1]{BBD}. This implies the isomorphisms of mixed $\Q$-Hodge structures
$$
H^i({a_U}_!\K^\bullet_U[-n])\isom H^i({a_U}_!\Q^H_U)\isom H^i_c(U,\Q)
$$
for $i\le n-1$.

The assumption $D_m$ for $U$ is equivalent to $\gr^F_{-p}\DR(\K^\bullet_U)=0$ for all $p\le m$. The dual statement of \cite[Lemma 3.4]{Park23} implies that
$$
\gr^F_{-p}\DR({a_U}_!\K^\bullet_U)=0
$$
for all $p\le m$; see also \cite[Lemma 3.2]{Park23} for the behavior the graded de Rham functor $\gr^F_p\DR(\, \cdot \,)$ upon dualizing. Recall now that the category of mixed Hodge modules on the point $\mathrm{pt}$ is equivalent to the category of mixed $\Q$-Hodge structures, so the functor $\gr^F_{-p}\DR(\, \cdot \,)$ is equivalent to the classical graded functor $\gr^F_{-p}(\, \cdot \,)$. Using the strictness of Hodge filtration (see e.g. \cite[Lemme 5]{Saito88}), we deduce that
$$
\gr^F_{-p}H^{p+q}({a_U}_!\K^\bullet_U[-n])=0
$$
for all $p\le m$.

On the other hand, by Proposition \ref{prop: weight of K}, $\K^\bullet_U[-n]$ is of weight $\le -1$. Consequently by \cite[4.5.2]{Saito90}, ${a_U}_!\K^\bullet_U[-n]$ is of weight 
$\le -1$, and so $H^{p+q}({a_U}_!\K^\bullet_U[-n])$ is of weight $\le p+q-1$. Therefore, by Proposition \ref{prop: range of nonzero grdr} (or its mixed $\Q$-Hodge structure version), we have
$$
\gr^F_{-p}H^{p+q}({a_U}_!\K^\bullet_U[-n])=0
$$
for all $p\ge p+q-1-m$, or equivalently $q \le m +1$.

In conclusion, when $p+q\le n-1$, we have $\gr^F_{-p}H^{p+q}_c(U,\C)=0$ for $p\le m$ or $q\le m+1$. This completes the proof.
\end{proof}

\begin{rmk}[{{\bf $m = -1, 0$}}]\label{rmk:m=-1}
Theorem \ref{thm:Lefschetz-krat} contains interesting information even for arbitrary varieties, or for those with rational singularities. For an arbitrary variety $X$, we can set 
$m=-1$. In this case, it says that the restriction
$$
H^{p,0}(X)\to H^{p,0}(D)
$$
is injective for $p\le n-1$. Equivalently, in the language of \cite{HJ14},  the restriction map of global sections of $h$-differentials
$$
H^0(X,\Omega_{X,h}^p)\to H^0(D,\Omega_{D,h}^p)
$$
is injective for $p\le n-1$. Furthermore, if $X\sm D$ has rational singularities, this restriction map is an isomorphism for $p\le n-2$, and for $ q \le n-2$ so is the restriction map 
$$
H^{0,q}(X)\to H^{0,q}(D).
$$
\end{rmk}

\section{Rational homology manifolds}
This chapter is devoted to one of the main applications of the techniques in this paper, namely the analytic characterization of the rational homology manifold condition, the ``obvious" situation in which one has full symmetry of the Hodge-Du Bois diamond. This also relies on results in the subsequent Appendix, where we address the closely related question of characterizing $\lcdef (X)$ as well, hence the range of applicability of results like the weak Lefschetz Theorem \ref{thm:Lefschetz-lcdef}.

\subsection{Rational homology manifolds and symmetry}
Recall that variety $X$ of dimension $n$ is a rational homology manifold iff the natural morphism $\Q_X^H [n] \to \IC_X^H$ is an isomorphism. In particular, $\Q_X[n]$ is a perverse sheaf, and therefore by Theorem \ref{thm:lcdef-top} we have that $\lcdef (X) = 0$. Moreover, when $X$ is projective, it is clear that $X$ has full symmetry of the Hodge-Du Bois diamond, as singular cohomology coincides with  intersection cohomology. 

In preparation for the next section (devoted to characterizing rational homology manifolds), we show that when $X$ has isolated singularities, or more generally when the rational homology manifold condition is known to hold away from a finite set, we have the following converse:

\begin{thm}\label{thm:RHM-symmetry}
Let $X$ be a projective variety, which is a rational homology manifold away from a finite set of points (e.g. with isolated singularities). Then $X$ is a rational homology manifold if and only if the Hodge-Du Bois diamond of $X$ has full symmetry.
\end{thm}

\begin{proof} 
We only need to show the ``if" implication, so we assume that we have full symmetry. Since $X$ is a rational homology manifold away from a finite set of points, the RHM-defect object $\K^\bullet_X$ is supported on this finite set. In particular, its cohomology $\H^i(\K^\bullet_X)$ is a direct sum of finitely many mixed Hodge modules associated to mixed Hodge structures $V^H_x$ for each point $x\in X$. Its hypercohomology $\HH^0(\H^i(\K^\bullet_X))$ is a direct sum of $V_x^H$, and vanishes otherwise: $\HH^{\neq0}(\H^i(\K^\bullet_X))=0$. This implies that $\K^\bullet_X$ is zero if and only if its every hypercohomology vanishes, which we now aim to prove.

To begin with, we have $\HH^i(\K_X^\bullet)=0$ for $i\ge 1$, since $\H^{i}(\K^\bullet_X)=0$ for $i\ge 1$. From the long exact sequence of hypercohomology associated to the triangle \eqref{eqn: Q to IC triangle}, we obtain
$$H^i(X,\Q) \simeq \mathit{IH}^i(X,\Q)\,\,\,\,\,\,{\rm  for ~all} \,\,\,\, i > n$$
and the natural surjection
$$
H^n(X,\Q) \twoheadrightarrow\mathit{IH}^n(X,\Q).
$$
Additionally, from the symmetry hypothesis, $H^i(X,\Q)$ is a pure Hodge structure of weight $i$ for all $i$; see e.g. Lemma \ref{lem:FL}. This implies the natural inclusion
$$
H^i(X,\Q)\hookrightarrow \mathit{IH}^i(X,\Q)
$$
for all $i$, by Weber's theorem (see Remark \ref{rmk:Weber}). Again from the symmetry hypothesis, we have $h^i(X)=h^{2n-i}(X)=Ih^{2n-i}(X)=Ih^{i}(X)$ when $i<n$. We thus have natural isomorphisms
$$H^i(X,\Q) \simeq \mathit{IH}^i(X,\Q)\,\,\,\,\,\,{\rm  for ~all} \,\,\,\, i.$$
Therefore, the hypercohomology  $\HH^i(\K_X^\bullet)$ vanishes for all $i$, which completes the proof.
\end{proof}

\subsection{Criteria for surfaces}
Theorem \ref{thm:RHM-symmetry} applies in particular to normal surfaces. This section serves as a warm-up, where in addition we characterize the RHM condition among normal Du Bois surfaces. We will then proceed to characterizing this condition among varieties with rational singularities in dimension three and higher.

%\noindent
%{\bf Normal projective surfaces.}
It is well known that surfaces with rational singularities are rational homology manifolds. Here we prove the following more precise result.

\begin{thm}\label{thm:RHM-surfaces}
Let $X$ be a normal projective surface. Then the following are equivalent:

\noindent
(i) $X$ is a rational homology manifold.

\noindent 
(ii) The Hodge-Du Bois diamond of $X$ satisfies full symmetry, i.e.
$$
\underline h^{p,q}(X)=\underline h^{q,p}(X)=\underline h^{2-p,2-q}(X)=\underline h^{2-q,2-p}(X)
$$
for all $0\le p,q\le 2$.

If in addition $X$ has Du Bois singularities, then they are also equivalent to:

\noindent
(iii) $X$ has rational singularities.\footnote{The equivalence between (i) and (iii) holds of course even when $X$ is not projective, for instance by passing to a compactification which is smooth along the boundary, and applying the theorem.}
\end{thm}

The first equivalence in Theorem \ref{thm:RHM-surfaces} is a special case of Theorem \ref{thm:RHM-symmetry}; indeed, since $X$ is normal, the singularities are isolated.

As mentioned earlier, surfaces with rational singularities are rational homology manifolds; this also follows from Corollary \ref{cor: lcdef of k-rational}. Assume now that $X$ is Du Bois and has full symmetry of the Hodge diamond. The fact that $X$ must have rational singularities follows from the Lemma below, a simple application of the Leray spectral sequence. Note that when $X$ is Du Bois, the identities in the hypothesis of the Lemma are equivalent to the symmetry
$$\h^{0, i} (X) = \h^{n, n-i} (X) \,\,\,\,\,\,{\rm for ~all} \,\,\,\,i \le 3,$$
since $\DB_X^n \simeq R f_* \omega_{\widetilde X}$ (always) and $\DB_X^0 \simeq \O_X$.

\begin{lem}\label{lem:vanishing-R1}
Let $X$ be a normal projective variety with isolated singularities, and let $f \colon \widetilde{X} \to X$ be a resolution of singularities such that $h^i (X, \O_X) = h^i (\widetilde{X}, \O_{\widetilde{X}})$ for all $i\le 3$. Then:
\begin{enumerate}
\item If $X$ is a surface, then $X$ has rational singularities, i.e. $R^1 f_* \O_{\widetilde{X}} = 0$.
\item If $X$ is a threefold with $H^2 (X, \O_X) = 0$, then $X$ has rational singularities, i.e. $R^1 f_* \O_{\widetilde{X}} = R^2 f_* \O_{\widetilde{X}} = 0$.
\end{enumerate}
\end{lem}
\begin{proof} 
In the Leray spectral sequence converging to $H^1 (\widetilde{X}, \O_{\widetilde{X}})$ we have two terms to consider. First, it is immediate 
that 
$$H^1 (X, \O_X) = E^{1, 0}_2 = E^{1, 0}_\infty.$$
On the other hand, 
$$E^{0, 1}_\infty = E^{0, 1}_3 = \ker \big(E^{0,1}_2 \to E^{2,0}_2\big) = \ker \big(\varphi\colon H^0 (X, R^1 f_*\O_{\widetilde{X}})
\to H^2 (X, \O_X )\big).$$
Therefore we have 
$$h^{0,1} (\widetilde{X}) = h^1 (X, \O_X) + \dim \ker (\varphi),$$
which together with the hypothesis implies $\ker (\varphi)= 0$.

In the spectral sequence converging to $H^2 (\widetilde{X}, \O_{\widetilde{X}})$ we have three terms to consider. 
First, 
$$E^{2, 0}_\infty = E^{2, 0}_3 = {\rm coker} \big(E^{0,1}_2 \to E^{2,0}_2\big) = {\rm coker}\big(\varphi\colon H^0 (X, R^1 f_*\O_{\widetilde{X}})
\to H^2 (X, \O_X )\big).$$
Next we have $E^{1,1}_2 = 0$ since the singularities are isolated.

When $X$ is a surface, we also have $E^{0,2}_2 =0$ for dimension reasons. This gives
$$h^2 (X, \O_X) = h^2 (\widetilde{X}, \O_{\widetilde{X}})  = \dim {\rm coker}(\varphi) \le h^2 (X, \O_X),$$
from which we deduce that $\varphi \equiv 0$. Combined with the fact that $\ker (\varphi)= 0$, we deduce that $R^1 f_*\O_{\widetilde{X}} =0$, which proves (1).

When $X$ is a threefold, a straightforward argument in the same vein shows that 
$$E^{0,2}_{\infty} \simeq H^0 (X, R^2 f_* \O_{\widetilde X}),$$
due to the equality $h^3 (X, \O_X) = h^3 (\widetilde{X}, \O_{\widetilde{X}})$. Thus we have 
$$h^2 (\widetilde{X}, \O_{\widetilde{X}})  = \dim {\rm coker} (\varphi) + h^0 (X, R^2 f_* \O_{\widetilde X}).$$
 But the extra assumption $H^2 (X, \O_X)=0$ implies that $\varphi \equiv 0$, hence $R^1 f_* \O_{\widetilde X} =0$ since we also know that ${\rm ker}(\varphi) = 0$.
It also gives that ${\rm coker} (\varphi) = 0$, so finally our assumption implies that $R^2 f_* \O_{\widetilde X} = 0$ as well.
\end{proof}

\subsection{Local Picard groups vs. local analytic $\Q$-factoriality}
This section is an interlude, where we relate local analytic $\Q$-factoriality to the objects considered in this paper. This serves as preparation 
for the statements in the subsequent sections. It is also useful for some of the arguments in the companion paper \cite{PP25}. 

For a closed point $x \in X$, we consider the local analytic divisor class group $\mathrm{Cl}(\O^{\rm an}_{X,x})$. We 
express this invariant in terms of the intersection complex of $X$, using a result of Flenner.

\begin{prop}
\label{prop: local defect of Q-factoriality}
Let $X$ be a variety of dimension $n$, with rational singularities at $x\in X$. Then, there exists an isomorphism
$$
\mathrm{Cl}(\O^{\rm an}_{X,x})\tensor_\Z \Q\isom\H^{-n+2}(\iota_x^*\IC_X),
$$
where $\mathrm{Cl}(\O^{\rm an}_{X,x})$ is the divisor class group of the analytic local ring $\O^{\rm an}_{X,x}$ and $\iota_x:x\to X$ is the closed embedding.
\end{prop}
\begin{proof}
In this setting, Flenner {\cite[Satz 6.1]{Flenner81}} establishes the isomorphism
$$
\mathrm{Cl}(\O^{\rm an}_{X,x})\isom \varinjlim_{x\in U} H^2(U\sm\Sing(X),\Z),
$$
where the direct limit is taken over all the analytic open neighborhoods of $x$ under inclusion. Denote $j:X\sm\Sing(X)\hookrightarrow X$. Sheaf theoretically, we have
$$
\mathrm{Cl}(\O^{\rm an}_{X,x})\isom\H^2(\iota_x^*Rj_*\Z_{X\sm \Sing(X)}),
$$
and therefore
$$
\mathrm{Cl}(\O^{\rm an}_{X,x})\tensor_{\Z}\Q\isom\H^2(\iota_x^*Rj_*\Q_{X\sm \Sing(X)}).
$$

By Corollary \ref{cor: lcdef of k-rational}, $X$ is a rational homology manifold away from a subvariety of codimension $3$. Therefore, there exists a Zariski closed subset $Z\subset \Sing(X)$ satisfying the following conditions:
\begin{enumerate}[(a)]
    \item $\codim_XZ\ge 3$,
    \item $\Sing(X)\sm Z$ is a smooth equidimensional variety of dimension $n-2$,
    \item $\IC_{X\sm Z}=\Q_{X\sm Z}[n]$.
\end{enumerate}
Let $j_Z:X\sm Z\to X$ be the open embedding. Using this setup, we prove the result in two steps.

\smallskip

\noindent
\emph{Step 1.} We establish an isomorphism of constructible cohomologies
$$
\H^2(Rj_*\Q_{X\sm \Sing(X)})\isom \H^2(R{j_Z}_*\Q_{X\sm Z}).
$$

Let $j':X\sm \Sing(X)\hookrightarrow X\sm Z$ and $i':\Sing(X)\sm Z\hookrightarrow X\sm Z$ be the open and closed embeddings. Consider the standard distinguished triangle
\begin{equation}
\label{eqn: Sing X and Z distinguished triangle}
i'_*{i'}^!\IC_{X\sm Z}\to \IC_{X\sm Z}\to Rj'_*\IC_{X\sm \Sing(X)}\xrightarrow{+1}. 
\end{equation}
Note that we have
$$
{i'}^!\IC_{X\sm Z} \simeq \dual({i'}^*\IC_{X\sm Z}) \simeq \dual(\Q_{\Sing(X)\sm Z}[n])\simeq \Q_{\Sing(X)\sm Z}[n-4].
$$
The first isomorphism comes from the definition of  ${i'}^!$ and the self-duality of intersection complexes, the second is a consequence of condition (c), and the third is a consequence of  condition (b). Therefore, taking the derived pushforward $R{j_Z}_*$ of \eqref{eqn: Sing X and Z distinguished triangle}, we get
$$
R{j_Z}_*\circ i'_*(\Q_{\Sing(X)\sm Z})[n-4]\to R{j_Z}_*\Q_{X\sm Z}[n]\to Rj_*\Q_{X\sm \Sing(X)}[n]\xrightarrow{+1},
$$
which implies that $\H^2(Rj_*\Q_{X\sm \Sing(X)})\isom \H^2(R{j_Z}_*\Q_{X\sm Z})$ from the long  exact sequence of cohomology.

\smallskip

\noindent
\emph{Step 2.} We next establish an isomorphism of constructible cohomologies
$$
\H^2(R{j_Z}_*\Q_{X\sm Z})\isom \H^{-n+2}(\IC_{X}).
$$

Let $i:Z\hookrightarrow X$ be the closed embedding and consider again the standard distinguished triangle
$$
i_*i^!\IC_X\to\IC_X\to R{j_Z}_*\Q_{X\sm Z}[n]\xrightarrow{+1},
$$
where we use condition (c). By \cite[Proposition 8.2.5]{HTT}, we have $i^!\IC_X\in {}^p\!D_c^{\ge 1}(Z,\Q)$, the bounded derived category of constructible $\Q$-sheaves with perverse cohomologies supported in degrees $\ge 1$. Since $\dim Z\le n-3$, we have the vanishing of constructible cohomologies
$$
\H^k(i^!\IC_X)=0 \,\,\,\,\,\, {\rm for} \,\,\,\,\,\, k\le -n+3.
$$
See \cite[Proposition 8.1.24]{HTT}. This implies that
$$
\H^2(R{j_Z}_*\Q_{X\sm Z})\isom \H^{-n+2}(\IC_{X}),
$$
from the long exact sequence of cohomology.

Combining Step 1 and Step 2,  and using the fact that  $\iota_x^*$ is an exact functor on the derived category of constructible sheaves, we conclude that
$$
\mathrm{Cl}(\O^{\rm an}_{X,x})\tensor_{\Z}\Q\isom\H^2(\iota_x^*Rj_*\Q_{X\sm \Sing(X)})\isom \H^{-n+2}(\iota_x^*\IC_X).
$$
\end{proof}

We record an immediate consequence of this proposition.

\begin{cor}\label{cor:RHM-Qfact}
If $X$ is an RHM with rational singularities, then $X$ is locally analytically $\Q$-factorial. 
\end{cor}

We may also consider the related object ${\rm Pic}^{\rm an-loc} (X, x)$, the local analytic Picard group at $x$. For a general discussion, see \cite[\S1]{Kollar-Picard}; when $\depth_x \O_X \ge 3$, it coincides with the analytic Picard group of a punctured neighborhood $U \smallsetminus \{x\}$, where $U$ is the intersection of $X$ with a small ball centered at $x$.
Therefore, in this case, we have 
$${\rm Pic}^{\rm an-loc} (X, x) \simeq H^1 (U \smallsetminus \{x\}, \O^*).$$

\begin{cor}\label{cor:equiv-threefold}
Let $X$ be a normal variety which is an RHM away from finitely many points. If $X$ has rational singularities, then for every $x \in X$ we have 
$${\rm Pic}^{\rm an-loc} (X, x)\otimes_{\Z} \Q \simeq {\rm Cl} (\O_{X, x}^{{\rm an}})\otimes_{\Z} \Q.$$
More generally, this holds when $R^i f_* \O_{\widetilde X} = 0$ for $i = 1,2$ with $f \colon \widetilde X \to X$ a resolution.
%In particular, the following are equivalent:
%\begin{enumerate}
%\item $X$ has torsion local analytic Picard groups.
%\item $X$ is locally analytically $\Q$-factorial.
%\end{enumerate}
\end{cor}
\begin{proof}
When $R^1f_*O_{\widetilde X}=0$, $X$ satisfies $S_3$; see \cite[Lemma 6.3]{PP25}, which is a result from \cite{Kollar}.
Therefore, this follows from Proposition \ref{prop: local defect of Q-factoriality} combined with Proposition \ref{prop:RHM-isolated} in the Appendix, 
showing that both sides are isomorphic to $\H^{-n+2}(\iota_x^*\IC_X)$. Note that, as in the Appendix, Proposition \ref{prop: local defect of Q-factoriality} works under the assumptions in the final part as well.
\end{proof}

\subsection{A criterion for threefolds}
\label{sec: Projective threefold with rational singularities}
In the case of threefolds, when aiming to characterize rational homology manifolds, Lemma \ref{lem:vanishing-R1} suggests that at least in the world of  Du Bois (so for instance log canonical) singularities, we need to focus on varieties with rational singularities.  However the situation is more complicated than in the case of surfaces, as the first local cohomological invariant comes into play via Theorem \ref{thm:RHM-characterization} in the 
Appendix.

\begin{proof}[{Proof of Theorem \ref{thm:RHM-threefolds}}]
When $X$ is projective, the equivalence between (i) and (iv) is a special case of Theorem \ref{thm:RHM-symmetry}, since 
$X$ is an RHM away from a finite set of points by Corollary \ref{cor: lcdef of k-rational}.

Again since $X$ is an RHM away from finitely many points,  the equivalence between (i) and (ii) is a consequence of the general characterization of RHMs in Theorem \ref{thm:RHM-characterization} in the Appendix.

Finally, the equivalence between (ii) and (iii) is a general property of all varieties that are RHMs away from finitely many points; see 
Corollary \ref{cor:equiv-threefold}.
\end{proof}

%\begin{rmk}
%Note in particular that the statement and proof are in line with the general fact that (local) analytic $\Q$-factoriality implies $\Q$-factoriality, and is typically a stronger condition.
%\end{rmk} 

\begin{rmk}
A more refined way of obtaining the equivalence between (ii) and (iv) in the projective case, by means of concrete formulas for the $\Q$-factoriality defect and local analytic $\Q$-factoriality defect in terms of the Hodge-Du Bois numbers, is explained in \cite{PP25}.
\end{rmk}

\subsection{Criteria for fourfolds and higher}
The pattern started in the previous two sections continues in higher dimension, in terms of the local cohomological invariants $H^i_L (X, x)$ 
introduced in the Appendix. We first record the following consequence of Theorem \ref{thm:RHM-threefolds}.

\begin{cor}\label{cor:rhm-fourfold-finite}
A variety $X$ with rational locally analytically $\Q$-factorial singularities is an RHM away from a closed subset of codimension at least $4$.
\end{cor}
\begin{proof}
We may assume that $X$ is quasi-projective. Note first that a general hyperplane section  $Y$ of $X$ is also locally analytically $\Q$-factorial, with rational singularities. The rational singularities part is well known, while the local analytic $\Q$-factoriality part follows from Proposition \ref{prop: local defect of Q-factoriality} and the behavior of the intersection complex with respect to taking general hyperplane sections; see the proof of Lemma \ref{lem: cutting hyperplane of K}.

The result follows then by cutting with general hyperplanes all the way down to the case of a threefold $Z$, which by Theorem \ref{thm:RHM-threefolds} is an RHM.
\end{proof}

\begin{proof}[{Proof of Theorem \ref{thm:RHM-fourfolds}}]
By Corollary \ref{cor:rhm-fourfold-finite}, $X$ is an RHM away from finitely many points. By Corollary \ref{cor:equiv-threefold}, local analytic $\Q$-factoriality is then equivalent to torsion local analytic Picard groups. Moreover, the local analytic gerbe groups being torsion, in the language of the Appendix, is equivalent to $X$ satisfying property $L_2$ along closed points.

The result then follows from the general Theorem \ref{thm:RHM-symmetry} and Theorem \ref{thm:RHM-characterization}.
\end{proof}

In arbitrary dimension, Theorem \ref{thm:RHM-characterization} in the Appendix provides a criterion for the RHM condition in terms of local analytic invariants; the interpretation of the higher such invariants (for $i \ge 3$) is however more mysterious. Combined with Theorem \ref{thm:RHM-symmetry}, we have the following:

\begin{cor}\label{cor:RHM-general}
If $X$ is a variety with rational singularities which is an RHM away from finitely many points, then $X$ is an RHM if and only if 
$$H^i (U \sm \{x\}, \O^*) \otimes_\Z \Q = 0  \,\,\,\,\,\,{\rm for~all}\,\,\,\, 1 \le i \le n-2,$$
for all $x \in X$, where $U$ is a sufficiently small analytic neighborhood of $x$. When $X$ is projective, this is moreover equivalent to the full symmetry of the
Hodge-Du Bois diamond of $X$.
\end{cor}

This subsumes the statements in Theorems \ref{thm:RHM-threefolds} and \ref{thm:RHM-fourfolds}.

\section{Appendix}
The appendix is devoted to characterizations of the local cohomological defect of a complex variety $X$, and of the RHM condition, in terms of local cohomological invariants. 
This is needed in the body of the paper, but is also of independent interest. We translate a topological characterization of the local cohomological dimension in \cite{RSW23} into a characterization in terms of the cohomology of $\O^*$ over a punctured neighborhood. In the case of $\lcdef (X)$, this extends the results of Ogus and Dao-Takagi to arbitrary dimension (or depth); see Remark \ref{rmk:DT}.

\subsection{Characterizations of $\lcdef (X)$ and of RHMs}\label{scn:appendix}
In this section, $X$ is a complex variety of dimension $n$. For all $i$, we denote:
$$
H^i_L (X, x) : =\varinjlim_{x\in U} H^i \left(U\sm \left\{x\right\},\O_{U\sm \left\{x\right\}}^*\right)
$$
where the direct limit is taken over all the analytic open neighborhoods of $x$ under inclusion. (Here the letter $L$ stands for ``local".)

\begin{defn}
We say that $X$ satisfies \textit{property $L_k$ along a closed subset $Z\subset X$} if there exists a stratification $Z=\sqcup_S S$ by locally closed subsets such that, for every stratum $S$ and every point $x\in S$,
$$H^i_L(X, x) \otimes_\Z \Q = 0 \,\,\,\,\,\,{\rm for~all}\,\,\,\, 1 \le i \le k + d_S,$$
where $d_S:=\dim S$.
\end{defn}

\begin{rmk}
Note that $H^1 _L (X, x) = {\rm Pic}^{\rm an-loc} (X, x)$, the local analytic Picard group of $X$ at $x$. Property $L_1$ along $Z = \{x\}$ is equivalent to these groups being torsion. 

Analogously, we can interpret $H^2_L (X, x)$ as the local analytic gerbe group of $X$ (meaning gerbes with band $\O_X^*$; see e.g. \cite[Theorem 5.2.8]{Brylinski}), and this has the same relationship with property $L_2$.
\end{rmk}

\noindent
{\bf Local cohomological defect.}
We have seen in Theorem \ref{thm:Lefschetz-lcdef} and Theorem \ref{thm:lcdef-top} that $\lcdef (X)$ is one of the main invariants governing the topology of $X$, the ideal situation 
being $\lcdef (X) = 0$. For this to happen, \cite{DT16} tells us that when $n \ge 4$ at the very least we need the local Picard groups of $X$ to be torsion, i.e. property $L_1$. In general, we have:

\begin{thm}\label{thm:lcdef-characterization2}
Let $X$ be an $n$-dimensional variety. Let $Z\subset X$ be a closed subset, satisfying
$$\dim Z \le \depth_x(\O_X) - k$$
for every closed point $x\in Z$. Assume that $\lcdef (X\sm Z) \le n-k$. Then $X$ satisfies property $L_{k-3}$ along $Z$ if and only if $\lcdef (X) \le n-k$. 
\end{thm}

As a special case, when $\lcdef (X) \le n-k$ and ${\rm depth} (\O_X) \ge k$, we have the property $L_{k-3}$ along $\{x\}$ for every point $x\in X$:
$$
H^i_L(X, x) \otimes_\Z \Q = 0 \,\,\,\,\,\,{\rm for~all}\,\,\,\, 1 \le i \le k-3.
$$
When $X$ has isolated singularities, the converse holds as well; we write this down separately, as it is a particularly clean statement that provided the intuition for the general picture: 

\begin{cor}
If $X$ has isolated singularities and ${\rm depth} (\O_X) \ge k$, then
$$\lcdef (X) \le n- k \iff X~{\rm satisfies}~L_{k-3}~{\rm along~each}~x \in X.$$
\end{cor}

Additionally, we obtain a similar criterion when we have partial vanishing of the higher direct images of the structure sheaf of a resolution.

\begin{thm}\label{thm:lcdef-characterization}
Let $X$ be an $n$-dimensional normal variety, and let $s$ be a positive integer with the property that for a resolution of singularities $f \colon \widetilde{X} \to X$ we have $R^i f_* \O_{\widetilde{X}} = 0$ for all $1\le i \le s$. Let $Z\subset X$ be a closed subset, and $k\ge 0$ an integer, satisfying
$$\dim Z \le s - k + 2.$$
Assume that $\lcdef (X\sm Z) \le n-k$. If $X$ satisfies property $L_{k-3}$ along $Z$, then $\lcdef (X) \le n-k$.  
\end{thm}

\begin{rmk}
1) For example, when $X$ has rational singularities we can take $s = n$, $Z = {\rm Sing}(X)$, and any $k \le n+1 - \dim  {\rm Sing}(X)$.\footnote{While to apply the theorem one could also take $k \le n+2 - \dim  {\rm Sing}(X)$, in fact property $L_{n-1}$ is never satisfied, and similarly for $L_{n-2}$ except possibly when $X$ has isolated singularities.}

\noindent
2) The converse to the statement in Theorem \ref{thm:lcdef-characterization} is already contained in Theorem \ref{thm:lcdef-characterization2}, since in order to ensure it one needs the depth condition appearing as the hypothesis of that theorem.
\end{rmk}

\begin{rmk}[{\bf Ogus and Dao-Takagi theorems}]\label{rmk:DT}

Theorem \ref{thm:lcdef-characterization2}  can be used to inductively generalize the best-known results on bounding the local cohomological dimension, due to Ogus and Dao-Takagi:

\begin{enumerate}
\item \cite{Ogus73} If $\depth(\O_X) \ge 2$, then $\lcdef (X) \le n-2$.
\item \cite{DT16} If $\depth(\O_X) \ge 3$, then $\lcdef (X) \le n-3$.
\item \cite{DT16} If $\depth(\O_X) \ge 4$, then $\lcdef (X) \le n-4 \iff X$ has torsion local analytic Picard groups.
\end{enumerate}
Note that using Theorem \ref{thm:lcdef-characterization2}, by induction we obtain (1) $ \Rightarrow$ (2) $\Rightarrow$ (3). Indeed, we may first assume $X$ is quasi-projective. If $\depth(\O_X) \ge 3$, then $\depth(\O_Y) \ge 2$ for a general hyperplane section $Y\subset X$, and $\lcdef (Y) \le n-3$ by (1). Consequently, $\lcdef (X\sm Z) \le n-3$ for some finite subset $Z\subset X$. Since property $L_0$ along $Z$ is vacuous, Theorem \ref{thm:lcdef-characterization2} yields (1) $ \Rightarrow$ (2). Likewise, if $\depth(\O_X) \ge 4$, then $\lcdef (X\sm Z) \le n-4$ for some finite subset $Z\subset X$. Therefore, Theorem \ref{thm:lcdef-characterization2} yields (2) $\Rightarrow$ (3).

Note that the condition ``the local analytic Picard groups are torsion" is not stable under taking hyperplane sections. By contrast, when the singularities are rational, local analytic $\Q$-factoriality is stable under taking general hyperplane sections. Thus the same inductive pattern continues:

%%By \cite[Corollary 2.11]{Ogus73}, $\lcdef (X) \le n-2$ is equivalent to the fact that the link at every $x\in X$ is connected.
\end{rmk}

\begin{cor}
Let $X$ be an $n$-dimensional variety, $n \ge 5$, with rational locally analytically $\Q$-factorial singularities. Then $\lcdef (X) \le n -5$ if and only if $X$ has torsion local analytic gerbe groups. 

In particular, if $X$ is a fivefold, then $\lcdef (X) = 0$ $\iff$ $X$ has torsion local analytic gerbe groups. 
\end{cor}

\begin{proof}
We may assume that $X$ is quasi-projective. As in the proof of Corollary \ref{cor:rhm-fourfold-finite}, if $Y$ is a general hyperplane section of $X$, then $Y$ also has rational locally analytically $\Q$-factorial singularities. In particular, $Y$ has torsion local analytic Picard groups, which implies $\lcdef (Y) \le n-5$ by the Dao-Takagi theorem mentioned above.

Consequently, $\lcdef (X\sm Z) \le n-5$ for some finite subset $Z\subset X$. Therefore, Theorem \ref{thm:lcdef-characterization2} concludes the proof.
\end{proof}

At present, we do not know whether the condition “the local analytic gerbe groups are torsion” is stable under taking general hyperplane sections, nor how to replace it by a comparably useful assumption known to be stable. This is the main obstruction to continuing this argument in order to obtain similar ``higher" statements inductively.

\medskip

\noindent
{\bf RHM condition.}
Along the same lines, we can give a similar criterion for the RHM condition. The rough guiding principle is that the RHM condition in dimension $n-1$ should be the same as the $\lcdef (X) = 0$ condition in dimension $n$.

\begin{thm}\label{thm:RHM-characterization}
Let $X$ be an $n$-dimensional variety with rational singularities and $Z\subset X$ be a finite subset of closed points. Assume that $X\sm Z$ is an RHM. Then $X$ satisfies property $L_{n-2}$ along $Z$ if and only if $X$ is an RHM. 
\end{thm}

Equivalently, when $X$ has rational singularities and is an RHM away from a point $x\in X$, then
$$
X \textrm{ is an RHM}\Longleftrightarrow
H_L^{\,i}(X,x)\otimes_{\mathbb Z}\mathbb Q=0\ \textrm{ for all } 1\le i\le n-2.
$$
In particular, if $X$ is an RHM, the above vanishing holds for every $x\in X$.

\medskip

\noindent
{\bf Proof of Theorems \ref{thm:lcdef-characterization2} and \ref{thm:lcdef-characterization}.}
Let $x \in X$, and assume that $x$ lies in a stratum $S$ of dimension $d_S$ of a Whitney stratification of $X$. If $T_x$ is a transversal slice
to $S$ at $x$, we can look at the link $L (T_x, x)$ of $T_x$ at $x$; it is well known that its cohomology depends only on the stratum $S$. Consider also a small analytic neighborhood $U$ of $x$ in $X$, which can be taken to be Stein and of product type, meaning of the
form $U = V \times B$, where $V$ is a (contractible) neighborhood of $x$ in $T_x$, and $B$ is a ball in $S$; see e.g. the proof of \cite[Theorem 4.1.10 or 5.1.20]{Dim04}. In particular, $U$ is contractible.

\begin{lem}\label{lem:link-top}
For every integer $i$, we have the following isomorphisms
$$\widetilde H^i (U  \sm \{x\}; \Q) \simeq \widetilde{H}^{i - 2d_S} (L(T_x, x), \Q) \simeq \widetilde H^{i-2d_S} (U  \sm S; \Q).$$
\end{lem} 
\begin{proof}
Since $U$ is contractible, we have isomorphisms
$$\widetilde H^i (U \sm \{x\}, \Q) \simeq H^{i+1} (U, U\sm \{x\}; \Q)$$
and similarly with $S$ instead of $\{x\}$. In this latter form, the statements are standard; the one for $S$ appears for instance in the course of the proof of \cite[Theorem 5.1.20]{Dim04} (with $T^*$ in \emph{loc. cit.} being homotopy equivalent to our $L( T_x, x)$), while the one for $x$ appears for instance in the course of the proof of \cite[Proposition 6.6.2]{Maxim19}.
\end{proof}

%% Note that $\widetilde H^i (U  \sm \{x\}; \Q)$ can be identified with $H^{i+1}_x(U)=H^{i+1}_x(B,\iota^!\Q_U)$ where $\iota:\{0\}\times B \hookrightarrow U$. The cohomology $\H^j(\iota^!\Q)$ is a local system on $B$ whose fibers are $\widetilde H^{j-1}(L(T_x,x),\Q)$. Note that the local cohomology of a local system $V$ on $B$ is zero, excluding $H_x^{2d_S}(B,V)=V_x$. This implies $H^{i+1}_x(U)=\widetilde H^{i-2d_S}(L(T_x,x),\Q).$

We use the following interpretation of the topological characterization of the local cohomological dimension in Theorem \ref{thm:lcdef-top},  in terms of link cohomology:

\begin{thm}[{\cite[Corollary 3]{RSW23}}]\label{thm:RSW}
For every $k\ge 0$, the condition $\lcdef (X) \le n - k$ is equivalent to the vanishing
$$\widetilde{H}^i  (L(T_x, x), \Q) = 0 \,\,\,\,\,\,{\rm for~all}\,\,\,\,i \le k - d_S - 2,$$
for every stratum $S$ of a Whitney stratification of $X$, and any $x \in S$, with the notation above.
\end{thm}

According to Lemma \ref{lem:link-top}, this can be reinterpreted as follows:

\begin{cor}\label{cor:lcdef-vanishing}
For every $k\ge 0$, the condition $\lcdef (X) \le n - k$ is equivalent to the vanishing
$$\widetilde H^i  (U  \sm \{x\}, \Q) = 0 \,\,\,\,\,\,{\rm for~all}\,\,\,\,i \le k + d_S - 2,$$
for every stratum $S$ of a Whitney stratification of $X$, and any $x \in S$, with the notation above.
\end{cor}

The main distinction between Theorem \ref{thm:RSW} and Corollary \ref{cor:lcdef-vanishing} is that the sign of $d_S$ is reversed. Since every stratification can be refined
 to a Whitney stratification, this allows one to check the vanishing condition for an arbitrary stratification, not necessarily Whitney:

\begin{prop}\label{prop:lcdef-vanishing}
Let $Z$ be a closed subset of $X$ with  $\lcdef(X\sm Z)\le n-k$ for some $k\ge 0$. Then $\lcdef (X) \le n-k$ if and only if there exists a stratification $Z=\sqcup_S S$ such that for every stratum $S$ and every $x\in S$, and for $U$ a sufficiently small Stein neighborhood of $x$, one has the vanishing
$$\widetilde H^i  (U  \sm \{x\}, \Q) = 0 \,\,\,\,\,\,{\rm for~all}\,\,\,\,i \le k + d_S - 2.$$
\end{prop}

\begin{proof}
Suppose $\lcdef (X) \le n-k$. Then a Whitney stratification that refines the stratification $X=(X\sm Z)\sqcup Z$ induces a stratification $Z=\sqcup_S S$ with the desired vanishing by Corollary \ref{cor:lcdef-vanishing}. Conversely, given a stratification $Z=\sqcup_S S$ with the vanishing, there exists a finer Whitney stratification $X=\sqcup_{S'} S'$ of $X=(X\sm Z)\sqcup (\sqcup_S S)$. Since $S'\subset S$ implies $d_{S'}\le d_S$, this finer Whitney stratification satisfies the vanishing in Corollary \ref{cor:lcdef-vanishing}.
\end{proof}

Next we compare the topological invariants with the local cohomological invariants. We use a standard method based on the exponential sequence, 
as well as the following important observation due to Flenner, which is a slight rephrasing of \cite[Lemma 6.2]{Flenner81}.

\begin{lem}\label{lem:Flenner}
Let $X$ be a normal complex analytic variety, and suppose that for a resolution of singularities $f \colon \widetilde{X} \to X$ we have $R^i f_* \O_{\widetilde{X}} = 0$ for all $1 \le i \le s$. 
Let $U$ be a Stein open subset of $X$, and $Y$ a closed subset of $U$. Then the natural map
$$H^i (U \sm Y, \Z) \to H^i (U \sm Y, \O)$$
is the zero map, for every $1\le i\le s$. In particular, if $X$ has rational singularities, then this holds for all $i\ge 1$.
\end{lem}

\begin{cor}\label{cor:L-link}
Let $x\in X$ be a point in a complex analytic variety $X$, and let $U\subset X$ be a sufficiently small Stein neighborhood of $x$.

\noindent
(i) If $1\le k\le {\rm depth}_{x} \O_X-3$, then we have an isomorphism
$$ H^k_L (X, x) \simeq H^{k + 1} (U \setminus \{x\}, \Z)$$
and a natural inclusion
$$ H^{k+1}_L (X, x) \hookrightarrow H^{k + 2} (U \setminus \{x\}, \Z).$$

\noindent
(ii) Suppose $X$ is normal and $R^i f_*\O_{\widetilde X}=0$ for all $1\le i\le s$, where $f:\widetilde X\to X$ is a resolution of singularities. Then for $1\le i \le s-1$ there is a short exact sequence
$$0 \to H^i (U \sm \{x\}, \O) \to H^i_L( X, x) \to  H^{i +1}  (U  \sm \{x\}, \Z) \to 0.$$
If $X$ has rational singularities, this holds for all $i\ge 1$, and in particular 
$$ H^i_L (X, x) \simeq H^{i + 1}  (U \setminus \{x\}, \Z) \,\,\,\,\,\,{\rm for} \,\,\,\,1\le i \le n-2.$$
\end{cor}

\begin{proof} 
We consider the cohomology exact sequence associated to the exponential sequence on $U \setminus \{x\}$:
$$\cdots \to H^i (U \setminus \{x\}, \O) \to H^i (U \setminus \{x\}, \O^*) \to H^{i+1}  (U \setminus \{x\}, \Z) \to H^{i + 1} (U \setminus \{x\}, \O)  \to  \cdots.$$
On the other hand, we have the usual exact sequence for local cohomology:
$$\cdots \to H^i_{\left\{x\right\}} (U, \O) \to H^i (U, \O) \to H^i (U \setminus \{x\}, \O) \to H^{i+1}_{\{x\}} (U, \O) \to \cdots.$$
By \cite[Theorem 1.14]{ST71}, the depth condition implies that $H^{i}_{\left\{x\right\}}(U,\O_U)=0$ for $i\le k+2$. This gives
$$
H^{i}\left(U\sm \left\{x\right\},\O\right) \simeq H^{i}(U,\O),
$$
for $1\le i\le k+1$, which is zero since $U$ is Stein. Plugging this into the first sequence gives (i).

For (ii) we use in addition Lemma \ref{lem:Flenner}, with $Y = \{x\}$, and the fact that a variety with rational singularities has maximal depth. 
\end{proof}

We now prove Theorems \ref{thm:lcdef-characterization2} and \ref{thm:lcdef-characterization}.

\begin{proof}[{Proof of Theorem \ref{thm:lcdef-characterization2}}]
From the assumption, we have $\depth_x(\O_X)\ge \dim Z+k$ for every point $x\in Z$, and Corollary \ref{cor:L-link} implies that
$$
H^i_L (X, x) \simeq H^{i + 1} (U \setminus \{x\}, \Z)
$$
for all $1\le i\le \dim Z+k-3$. Therefore, the vanishing in Proposition \ref{prop:lcdef-vanishing} is equivalent to the vanishing
$$
H^i_L (X, x)\tensor_\Z\Q=0\,\,\,\,\,\,{\rm for~all}\,\,\,\,1\le i \le k + d_S - 3
$$
and
\begin{equation}
\label{eqn:residual L-vanishing}
\widetilde H^i  (U  \sm \{x\}, \Q) = 0 \,\,\,\,\,\,{\rm for~all}\,\,\,\,i \le \min\{1,k + d_S - 2\}.
\end{equation}
This implies that $\lcdef(X)\le n-k$ is equivalent to property $L_{k-3}$ along $Z$ plus the vanishing \eqref{eqn:residual L-vanishing}. It suffices to prove that property $L_{k-3}$ along $Z$ implies $\lcdef(X)\le n-k$.

When $k\ge 3$, we have $\lcdef(X)\le n-3$ by \cite{DT16}, and thus, \eqref{eqn:residual L-vanishing} always holds. When $k\le 2$, we always have $\lcdef(X)\le n-k$ by \cite{Ogus73} and the local Lichtenbaum theorem (\cite[Theorem 3.1]{Hartshorne68} or \cite[Corollary 2.10]{Ogus73}).
\end{proof}

\begin{proof}[{Proof of Theorem \ref{thm:lcdef-characterization}}]
From the assumption, Corollary \ref{cor:L-link} implies that there is a surjection
$$
H^i_L( X, x) \twoheadrightarrow  H^{i +1}  (U  \sm \{x\}, \Z) 
$$
for all $1\le i \le s-1$. Therefore property $L_{k-3}$ implies the vanishing 
$$
\widetilde H^{i}  (U  \sm \{x\}, \Q)=0 \,\,\,\,\,\,{\rm for~ all}\,\,\,\, 2\le i\le k+d_S-2.
$$
Since $R^1 f_* \O_{\widetilde{X}} = 0$, $X$ satisfies $S_3$; see \cite[Lemma 6.3]{PP25}, which is a result from \cite{Kollar}. Again, by \cite{DT16}, we have $\lcdef(X)\le n-3$ and thus the above vanishing holds in fact for all $i\le k+d_S-2$. Proposition \ref{prop:lcdef-vanishing} then completes the proof.
\end{proof}

\noindent
{\bf Proof of Theorem \ref{thm:RHM-characterization}.} Recall that if $n\le 2$, the statement is vacuous, since a surface with rational singularities is an RHM. We prove the following proposition, which gives a precise description of the stalks of the intersection complex in terms of the local cohomological invariants $H^i_L(X,x)$. This immediately implies Theorem \ref{thm:RHM-characterization} and is also used in Corollary \ref{cor:equiv-threefold} in the body of the paper. 

\begin{prop}\label{prop:RHM-isolated}
Assume $X$ is normal, and $X\sm \left\{x\right\}$ is a rational homology manifold. If there exists an integer $2\le s\le \mathrm{depth}(\O_X)-1$ such that
$$
R^i\mu_*\O_{\widetilde X}=0\,\,\,\,{\rm for~ all } \,\,\,\, 1\le i\le s
$$
for a resolution of singularities $\mu:\widetilde X\to X$, then there exists an isomorphism
$$
H^{s-1}_L (X, x)\tensor_\Z\Q\simeq \H^{-n+s}(\iota_x^*\IC_X)
$$
where $\iota_x \colon \{x\} \hookrightarrow X$ is the closed embedding.
\end{prop}
\begin{proof}
This is similar to the proof of Proposition \ref{prop: local defect of Q-factoriality}. As in Corollary \ref{cor:L-link}, we obtain a short exact sequence
$$
0\to\varinjlim_{x\in U} H^{s-1}\left(U\sm \left\{x\right\},\O\right)\to H^{s-1}_L(X, x) \to \varinjlim_{x\in U} H^{s}\left(U\sm \left\{x\right\},\Z\right)\to 0.
$$
By \cite[Theorem 1.14]{ST71}, the depth condition implies that $H^{i}_{\left\{x\right\}}(U,\O_U)=0$ for $i\le s$. This gives
$$
H^{s-1}\left(U\sm \left\{x\right\},\O\right) \simeq H^{s-1}(U,\O),
$$
which is zero if $U$ is Stein. Therefore, we have
$$
H^{s-1}_L(X, x) \simeq \varinjlim_{x\in U} H^{s}\left(U\sm \left\{x\right\},\Z\right) \simeq \H^s (\iota_x^* Rj_*\Z_{X\sm\left\{x\right\}}) .
$$
Therefore it suffices to establish an isomorphism 
$$
R^i j_*\Q_{X\sm\left\{x\right\}} \simeq \H^{-n+i} (\IC_X) 
$$
for all $i\le n-1$, where $j:X\sm\left\{x\right\}\to X$ is the open embedding.
Note now that we have a distinguished triangle 
$$
{\iota_x}_* \iota_x^!\IC_X\to\IC_X\to R j_*\Q_{X\sm \{x\}}[n]\xrightarrow{+1}.
$$
By \cite[Proposition 8.2.5]{HTT}, we have $\iota_x^!\IC_X\in {}^p\!D_c^{\ge 1}(\{x\},\Q)$, hence
$$
\H^k(\iota_x^!\IC_X)=0 \,\,\,\,\,\, {\rm for} \,\,\,\,\,\, k\le 0
$$
by \cite[Proposition 8.1.24]{HTT}. Passing to cohomology in the distinguished triangle then completes the proof. 
\end{proof}

\bibliographystyle{alpha}
\bibliography{refs}

\end{document}